\documentclass[12pt,reqno]{amsart}
\usepackage{graphicx}
\usepackage{amssymb,amsmath}
\usepackage{amsthm}
\usepackage{color,graphicx}
\usepackage{hyperref}
\usepackage{color}
\usepackage{epstopdf}
\usepackage{lpic}

\newcommand{\be}{\begin{equation}}

\newcommand{\ee}{\end{equation}}

\numberwithin{equation}{section}
\numberwithin{figure}{section}

\newtheorem{theorem}{Theorem}[section]
\newtheorem{proposition}[theorem]{Proposition}
\newtheorem{remark}[theorem]{Remark}
\newtheorem{lemma}[theorem]{Lemma}
\newtheorem{corollary}[theorem]{Corollary}
\newtheorem{definition}[theorem]{Definition}

\begin{document}

\title[Smooth periodic waves in the Camassa-Holm equation]{Stability of smooth periodic traveling waves in the Camassa-Holm equation}

\author{Anna Geyer}
\address[A. Geyer]{Delft Institute of Applied Mathematics, Faculty Electrical Engineering, Mathematics and
	Computer Science, Delft University of Technology, Mekelweg 4, 2628 CD Delft, The Netherlands}
\email{A.Geyer@tudelft.nl}

\author{Renan H. Martins}
\address[R.H. Martins]{Departamento de Matem\'{a}tica - Universidade Estadual de Maring\'{a}, Avenida Colombo 5790, CEP 87020-900, Maring\'{a}, PR, Brazil}
\email{r3nan$\_$s@hotmail.com}

\author{F\'{a}bio Natali}
\address[F. Natali]{Departamento de Matem\'{a}tica - Universidade Estadual de Maring\'{a}, Avenida Colombo 5790, CEP 87020-900, Maring\'{a}, PR, Brazil}
\email{fmanatali@uem.br}

\author{Dmitry E. Pelinovsky}
\address[D.E. Pelinovsky]{Department of Mathematics and Statistics, McMaster University,
	Hamilton, Ontario, Canada, L8S 4K1}
\email{dmpeli@math.mcmaster.ca}

\subjclass[2000]{76B15, 37K45, 35Q53.}

\keywords{Camassa Holm equation, periodic traveling waves, spectral stability}

\begin{abstract}
	Smooth periodic travelling waves in the Camassa--Holm (CH) equation are revisited. We show that these periodic waves can be characterized in two different ways by using two different Hamiltonian structures. The standard formulation, common to the Korteweg--de Vries (KdV) equation, has several disadvantages, e.g., the period function is not monotone and the quadratic energy form may have two rather than one negative eigenvalues. We explore the nonstandard formulation and prove that the period function is monotone and the quadratic energy form has only one simple negative eigenvalue. We deduce a precise condition for the spectral and orbital stability of the smooth periodic travelling waves and show numerically that this condition is satisfied in the open region where the smooth periodic waves exist.
\end{abstract}

\date{\today}
\maketitle

\section{Introduction}

The Camassa--Holm (CH) equation
\begin{equation}
\label{CH}
u_t-u_{txx}+3uu_x=2u_xu_{xx}+uu_{xxx}
\end{equation}
 was derived in \cite{CH,CHH} and justified in \cite{CL,Johnson} as a model for the propagation of unidirectional shallow water waves. A generalized version of this equation also models propagation of nonlinear waves inside a cylindrical hyper-elastic rod with a small diameter \cite{Dai}. The CH equation can be interpreted geometrically in terms of geodesic flows on the diffeomorphism group \cite{Kouranbaeva,Mis1}.

We consider the CH equation (\ref{CH}) on the periodic domain $\mathbb{T}_L := [0,L]$ of length $L > 0$. For notational simplicity, we write 
$H^s_{\rm per}$ instead of $H^s_{\rm per}(\mathbb{T}_L)$.
The CH equation (\ref{CH}) on $\mathbb{T}_L$ 
conserves formally the mass, momentum, and energy 
given respectively by 
\begin{equation}
\label{Mu}
M(u)=\int_0^L u dx,
\end{equation}
\begin{equation}
\label{Eu}
E(u)=\frac{1}{2} \int_0^L (u_x^2+u^2) dx,
\end{equation}
and
\begin{equation}
\label{Fu}
F(u)=\frac{1}{2} \int_0^L (u^3+uu_x^2) dx.
\end{equation}

Many results are available for the CH equation (\ref{CH}) in the periodic domain $\mathbb{T}_L$. The initial-value problem is locally well-posed in the space $H^3_{\rm per}$ \cite{CE,CE2000}, 
$H^s_{\rm per}$ with $s > \frac{3}{2}$ \cite{CE-1998,hakka,HM2002}, $C^1_{\rm per}$ \cite{Mis2}, and $H^1_{\rm per} \cap {\rm Lip}$ \cite{DKT}, where ${\rm Lip}$ stands for Lipshitz continuous functions suitable for peaked periodic waves. 

Smooth, peaked, and cusped periodic travelling waves were classified in \cite{Len3}. Cusped periodic waves were constructed in \cite{B,Him} 
to show that local solutions in $H^1_{\rm per}$ are not uniformly continuous 
with respect to the initial data. More recently, the period function for the smooth periodic waves was analyzed in \cite{GV}. 
 
Orbital stability of the smooth periodic travelling waves in $H^1_{\rm per}$ was obtained in \cite{Len4} with the inverse scattering transform for initial data 
$u_0$ in $H^3_{\rm per}$ such that $m_0 := u_0 - u_0''$ is strictly positive. Thanks to the Lax representation, the orbital stability of the smooth periodic waves in the time evolution of the CH equation (\ref{CH}) follows from the structural stability of the Floquet spectrum in the associated Lax equations. 

Orbital stability of peaked periodic waves in $H^1_{\rm per}$ was proven in  \cite{Len1,Len2} by using two different variational methods, each uses the three conserved quantities (\ref{Mu}), (\ref{Eu}), and (\ref{Fu}). The recent work \cite{MP20} shows that the perturbations to the peaked periodic waves grow exponentially in $H^1_{\rm per} \cap W^{1,\infty}$ and may blow up in finite time. 

Stability of cusped periodic waves is an open problem due to the lack of continuity with respect to initial data in $H^1_{\rm per}$.

Regarding the limit to the solitary waves, orbital stability of smooth solitary waves in $H^1(\mathbb{R})$ was obtained in \cite{CS2} for a modified version of the CH equation, where the standard orbital stability for solitary waves hold.   
Orbital stability of peaked solitary waves in $H^1(\mathbb{R})$ was obtained 
in \cite{CM,CS} but the recent work \cite{NP} showed that the perturbations 
to the peaked solitary waves actually grow in $W^{1,\infty}(\mathbb{R})$.

{\em The main purpose of this paper} is to address the spectral and orbital stability of the smooth periodic waves by using the analytic theory used for stability of periodic waves in other nonlinear evolution equations of KdV type \cite{ANP,GP1,haragus,HJ,J,NLP}. 

In the standard spectral stability theory, we identify the smooth periodic travelling wave as a critical point of the action functional and compute the number of negative eigenvalues in the linearized operator which represents the quadratic energy form. The number of negative eigenvalues is typically controlled by the monotonicity of the period function for the smooth periodic waves \cite{J}. If the action functional is a linear combination of the mass, momentum, and energy, the spectral stability is determined by positivity of the linearized operator under the constraints of fixed mass and momentum. If the periodic waves of a fixed period are extended smoothly with respect to parameters, derivatives of the mass and momentum computed at the periodic waves  with respect to their parameters determine the number of negative eigenvalues of the linearized operator under the constraints.

The technique is straightforward in the case when the linearized operator has only one simple negative eigenvalue, e.g. in \cite{GP1}. Moreover, orbital 
stability of such periodic waves in the energy space can be easily concluded from its spectral stability \cite{ANP}. However, computations become messy 
when the linearized operator has two negative eigenvalues and the periodic wave has limited smoothness with respect to its parameters \cite{HJ,NLP}.

{\em The main novelty of this paper} is to show that stability of the smooth periodic waves of the CH equation (\ref{CH}) can be characterized in two different ways by using two different Hamiltonian structures \cite{CCQ}. The two equivalent formulations are related to two different parameters $a$ and $b$ of the periodic travelling wave solutions in addition to the wave speed $c$. The standard formulation common to evolution equations 
of KdV type has many disadvantages, whereas the nonstandard formulation common to the evolution equations of CH type is suitable for the proof of spectral 
and orbital stability of the smooth periodic waves. 
The results obtained here for the CH equation (\ref{CH}) could be applicable 
to other evolution equations of CH type such as the Degasperis--Procesi 
equation or the $b$-family of the CH equations \cite{HGH}.

We shall now describe {\em the main results of this paper}.

Smooth traveling waves of the form $u(x,t) = \phi(x-ct)$ with speed $c$
and profile $\phi$ satisfy the third-order differential equation 
\begin{equation}
\label{third-order}
-(c-\phi) (\phi''' - \phi') - 2 \phi \phi' + 2 \phi' \phi'' = 0.
\end{equation}
This equation can be integrated in two different ways. The standard integration of (\ref{third-order}) in $x$ gives the second-order equation:
\begin{equation}\label{CHode}
-(c-\phi)\phi''+c\phi-\frac{3}{2}\phi^2+\frac{1}{2} \phi'^2 = b,
\end{equation}
where $b$ is the integration constant. However, another integration is obtained after multiplying (\ref{third-order})  by $(c-\phi)$, 
which yields the second-order equation:
\begin{equation}
\label{second-order}
-(c - \phi)^2 (\phi'' - \phi) = a,
\end{equation}
where $a$ is another integration constant. Both second-order equations (\ref{CHode}) and (\ref{second-order}) are compatible if and only if 
$\phi$ satisfies the first-order invariant:
\begin{equation}
\label{quadra}
(c-\phi)(\phi'^2 - \phi^2 - 2 b) + 2a = 0.
\end{equation}
It is easy to verify that one of the three equations 
(\ref{CHode}), (\ref{second-order}), and (\ref{quadra}) 
is satisfied if and only if the other two equations are satisfied. 
We consider the smooth $L$-periodic travelling wave solutions of \eqref{CH}, 
which means that we are looking for solutions 
$\phi \in H^{\infty}_{\rm per}$ of the system (\ref{CHode}), (\ref{second-order}), and (\ref{quadra}).

Let us connect the second-order equations (\ref{CHode}) and (\ref{second-order})
with two different Hamiltonian structures of the CH equation (\ref{CH}). 

The standard Hamiltonian structure for the CH equation (\ref{CH}) is given by 
\begin{equation}
\label{sympl-1}
\frac{du}{dt} = J \frac{\delta F}{\delta u}, \quad 
J = -(1-\partial_x^2)^{-1} \partial_x, \quad 
\frac{\delta F}{\delta u} = \frac{3}{2} u^2 - u u_{xx} - \frac{1}{2} u_x^2,
\end{equation}
where $J$ is a well-defined operator from $H^s_{\rm per}$ to $H^{s+1}_{\rm per}$ for every $s \in \mathbb{R}$ and $\frac{\delta F}{\delta u}$ is an operator from $H^{s}_{\rm per}$ to $H^{s-2}_{\rm per}$ for $s > \frac{3}{2}$ thanks to Sobolev's embedding of $H^s_{\rm per}$ into $C^1_{\rm per}$. The evolution problem (\ref{sympl-1}) is well-defined for local solutions  
$u \in C((-t_0,t_0),H^s_{\rm per}) \cap C^1((-t_0,t_0),H^{s-1}_{\rm per})$ with some $t_0 > 0$ and  $s > \frac{3}{2}$, see \cite{CE-1998,hakka,HM2002}.

The second-order equation (\ref{CHode}) is the Euler--Lagrange equation for the action functional 
\begin{equation}
\label{action-1}
\Lambda_{c,b}(u) := c E(u) - F(u) - b M(u).
\end{equation}
The linearized operator for the second-order equation (\ref{CHode}) 
is given by 
\begin{equation}
\mathcal{L} := - \partial_x (c-\phi) \partial_x + (c - 3\phi + \phi''),
\label{hill} 
\end{equation} 
which is related to the action functional (\ref{action-1}) as $\mathcal{L}= \Lambda_{c,b}''(\phi)$. The linearized operator $\mathcal{L} : H^2_{\rm per} \subset L^2_{\rm per} \mapsto L^2_{\rm per}$ is a self-adjoint, unbounded operator in $L^2_{\rm per}$ equipped with the standard inner product $\langle \cdot, \cdot \rangle$. 

The alternative Hamiltonian structure for the CH equation (\ref{CH}) 
is given by 
\begin{equation}
\label{sympl-2}
\frac{d m}{dt} = J_m \frac{\delta E}{\delta m}, \quad J_m = -\left( m \partial_x + \partial_x m \right), \quad 
\frac{\delta E}{\delta m} = u, 
\end{equation}
where $m := u - u_{xx}$ and $E(u)$ can be written equivalently as 
\begin{equation}
\label{Em}
E(u) = \frac{1}{2} \int_0^L (u_x^2+u^2) dx = \frac{1}{2} \int_0^L u m dx = 
\frac{1}{2} \int_0^L m (1-\partial_x^2)^{-1} m dx.
\end{equation}
By using (\ref{sympl-2}), the CH equation (\ref{CH}) 
is rewritten in the local differential form
\begin{equation}\label{CHm}
m_t + u m_x + 2m u_x = 0.
\end{equation}
The second-order equation (\ref{second-order}) is related to the action functional 
\begin{equation}
\label{action-2}
\Lambda_c(m) := E(u) - c M(u), \quad u := (1-\partial_x^2)^{-1} m.
\end{equation}
Indeed, the Euler--Lagrange equation 
$J_m \frac{\delta \Lambda_c}{\partial m} = 0$  gives 
the differential equation 
\begin{equation}
\label{m-phi-connection}
\mu'(\phi - c) + 2 \mu \phi' = 0,
\end{equation}
where $\mu = \phi - \phi''$ or $\phi = (1-\partial_x^2)^{-1} \mu$.
Integration of (\ref{m-phi-connection}) multiplied by $(\phi-c)$ yields $(c-\phi)^2 \mu = a$ which is equivalent to (\ref{second-order}). The linearized operator for $(c-\phi)^3 \mu = a(c-\phi)$ acting on $\mu$ is given by 
\begin{equation}
\mathcal{K} := (c - \phi)^3 - 2a (1 - \partial_x^2)^{-1},
\label{hill-m} 
\end{equation} 
The linearized operator $\mathcal{K} : L^2_{\rm per} \mapsto L^2_{\rm per}$ 
is the sum of a bounded and a compact self-adjoint operator in $L^2_{\rm per}$. 
The relation of this operator to the action functional (\ref{action-2}) 
is not obvious and will be shown in Lemma \ref{lem-equivalency}.

Let us now give the definitions of spectral and orbital stability 
of the smooth periodic travelling waves in the CH equation (\ref{CH}). 

\begin{definition}\label{defstab-spectral}
	We say that the smooth periodic travelling wave $\phi \in H^{\infty}_{\rm per}$ is spectrally stable in the evolution problem (\ref{CH}) if the spectrum of $J \mathcal{L}$ in $L^2_{\rm per}$ 
	is located on the imaginary axis. 
\end{definition}

\begin{definition}\label{defstab}
	We say that the smooth periodic travelling wave $\phi \in H^{\infty}_{\rm per}$ is orbitally stable in the evolution problem (\ref{CH}) in $H_{\rm per}^1$ if for any $\varepsilon>0$ there exists $\delta>0$ such that for any $u_0 \in H^s_{\rm per}$ with $s > \frac{3}{2}$ satisfying 
	\begin{equation*}
		\|u_0-\phi\|_{H_{\rm per}^1}<\delta,
	\end{equation*}
the global solution $u \in C(\mathbb{R},H^s_{\rm per})$ with the initial data $u_0$ 
satisfies
	\begin{equation*}
	\inf_{r\in{\mathbb R}} \| u(t,\cdot) - \phi(\cdot+r) \|_{H_{\rm per}^1} <\varepsilon
	\end{equation*}
	for all $t\geq0$.
\end{definition}

The following two theorems represent the main results of this paper.

\begin{theorem}
	\label{theorem-existence}
For a fixed $c > 0$, smooth periodic solutions of the system (\ref{CHode}), (\ref{second-order}), and (\ref{quadra}) exist in an open, simply connected region on the $(a,b)$ plane closed by three boundaries: 
	\begin{itemize}
		\item $a = 0$ and $b \in (-\frac{1}{2} c^2,0)$ (where the periodic solutions are peaked), 
		\item $a = a_+(b)$ and $b \in (0,\frac{1}{6} c^2)$ (where the solutions have infinite period),
		\item $a = a_-(b)$ and $b \in (-\frac{1}{2} c^2,\frac{1}{6} c^2)$ (where the solutions are constant), 
	\end{itemize}
where $a_+(b)$ and $a_-(b)$ are  smooth functions of $b$ specified in Lemmas \ref{remark-right} and \ref{remark-upper}. For every point inside the region, the periodic solutions are smooth functions of $(a,b,c)$ and their period is strictly increasing in $b$ for every fixed $a \in (0,\frac{4}{27} c^3)$ and $c > 0$. Moreover, $\mathcal{K}$ has exactly one simple negative eigenvalue, a simple zero eigenvalue, and the rest of its spectrum in $L^2_{\rm per}$ is strictly positive and bounded away from zero.
\end{theorem}

\begin{theorem}
	\label{theorem-stability}
	For a fixed $c > 0$ and a fixed period $L > 0$, there exists 
	a $C^1$ mapping $a \mapsto b = \mathcal{B}_L(a)$ for $a \in (0,a_L)$ with some $a_L \in (0,\frac{4}{27} c^3)$ 
	and a $C^1$ mapping $a \mapsto \phi = \Phi_L(\cdot,a) \in H^{\infty}_{\rm per}$ of smooth $L$-periodic solutions along the curve $b = \mathcal{B}_L(a)$.	Let 
	$$
	\mathcal{M}_L(a) := M(\Phi_L(\cdot,a)) \quad \mbox{\rm and} \quad 
	\mathcal{E}_L(a) := E(\Phi_L(\cdot,a)). 
	$$
	The $L$-periodic wave with profile $\Phi_L(\cdot,a)$ is spectrally and orbitally stable in the sense of Definitions 
	\ref{defstab-spectral} and \ref{defstab} respectively, if the mapping
\begin{equation}
\label{stability-criterion}
	a \mapsto \frac{\mathcal{E}_L(a)}{\mathcal{M}_L(a)^2}
\end{equation}
	is strictly decreasing.
\end{theorem}

\begin{remark}
	\label{remark-existence}
	In comparison with Theorem \ref{theorem-existence}, it follows from the  results in \cite{GV} that 
	the period of the periodic solutions of the system (\ref{CHode}), (\ref{second-order}), and (\ref{quadra})
		\begin{itemize}
		\item is monotonically increasing in $a$ if $b \in (-\frac{1}{2} c^2,-(1 - \frac{\sqrt{2}}{\sqrt{3}}) c^2]$;
		\item has a single maximum point in $a$ if $b \in (-(1 - \frac{\sqrt{2}}{\sqrt{3}}) c^2,0)$; 
		\item is monotonically decreasing in $a$ if $b \in [0,\frac{1}{6} c^2)$.
	\end{itemize}
We will show in Theorem \ref{theolinop} that $\mathcal{L}$ has two simple negative eigenvalues if the period is increasing in $a$ and one simple negative eigenvalue if the period is decreasing in $a$, in addition to the simple zero eigenvalue in $L^2_{\rm per}$.
\end{remark}

\begin{remark}
	\label{remark-spectrum}
	The spectrum of $\mathcal{L}$ in $L^2_{\rm per}$ is purely discrete, whereas the spectrum of $\mathcal{K}$ in $L^2_{\rm per}$ includes both continuous and discrete parts, see Lemma \ref{lemma-L-K}. However, the continuous part of $\mathcal{K}$ is strictly positive and does not contribute to the count of negative eigenvalues. 
\end{remark}

\begin{remark}
	\label{remark-stability}
	In the context of Theorem \ref{theorem-stability}, we show numerically that the stability criterion (\ref{stability-criterion}) is satisfied for every periodic solution of Theorem \ref{theorem-existence}, see Figure \ref{fig-dependence}.
\end{remark}

\begin{remark}
	\label{remark-criterion}
	We also show in Lemma \ref{lemma-b-neg} and Remark \ref{remark-boundaries} without appealing to the numerical verification of the stability criterion (\ref{stability-criterion}) that the smooth periodic travelling waves of Theorem \ref{theorem-existence} are orbitally 
	stable for $b \leq 0$ and also in a neighborhood of the boundary $a = a_-(b)$ where the solutions are constant.
\end{remark}

\begin{remark}
	\label{remark-scaling}
	The following transformation 
	\begin{equation}
	\label{scal-transform}
	\phi(x) = c \varphi(x), \quad b = c^2 \beta, \quad a = c^3 \alpha
	\end{equation}
	normalizes the parameter $c$ to unity, so that $\varphi$, $\beta$, and $\alpha$ satisfy the same system  (\ref{CHode}), (\ref{second-order}), and (\ref{quadra})  but with $c = 1$. Hence, the smooth periodic waves are uniquely determined by the free parameters $(a,b)$ 
	and $c = 1$ can be used everywhere. For clarity of presentation, 
we will keep the parameter $c$ in all equations until Lemma \ref{lem-positivity},
where we will use the scaling transformation (\ref{scal-transform}).
\end{remark}

\begin{remark}
	\label{remark-positivity}
We only consider the case of right-propagating waves with $c > 0$; 
however, all results are extended to the left-propagating waves with $c < 0$ 
by simply flipping the signs in the scaling transformation (\ref{scal-transform}).
\end{remark}

The paper is organized as follows. In Section \ref{sec-2}, we study the existence of the smooth periodic solutions and give the proof of 
the first two assertions in Theorem \ref{theorem-existence}.
In Section \ref{sec-3}, we study the eigenvalues of the linearized 
operators $\mathcal{L}$ and $\mathcal{K}$ and give the proof 
of the last assertion of Theorem \ref{theorem-existence}. 
In parallel, we also show the assertions in Remarks \ref{remark-existence} 
and \ref{remark-spectrum}. 
In Section \ref{sec-4}, we study the linearized evolution 
under two constraints and give the proof of spectral stability 
in Theorem \ref{theorem-stability} and Remark \ref{remark-stability}.
In Section \ref{sec-5}, we prove the orbital stability 
in Theorem \ref{theorem-stability} 
and also obtain orbital stability directly for $b \leq 0$ as in Remark \ref{remark-criterion}. 
Section \ref{sec-6} concludes the paper with 
a summary and a discussion of open directions. 
Appendix A describes the approach used for numerical approximations 
of the smooth periodic waves.

\section{Existence of smooth periodic traveling waves}
\label{sec-2}

Here we study existence of smooth periodic solutions of the system (\ref{CHode}), (\ref{second-order}), and (\ref{quadra}) and 
provide the proof of the first two assertions of Theorem \ref{theorem-existence}. 

Let us rewrite (\ref{quadra}) as the total energy $b$ of Newton's particle of unit mass with the coordinate $\phi$ in ``time" $x$ with the potential energy $U(\phi)$:
\begin{equation}
\label{first-order}
b = \frac{1}{2} \left( \frac{d\phi}{dx} \right)^2 + U(\phi), \qquad 
U(\phi) = -\frac{1}{2} \phi^2 + \frac{a}{c - \phi}
\end{equation}
Critical points of $U$ are given by roots of the cubic equation $a = \phi (c-\phi)^2$. The local maximum of $\phi \mapsto \phi (c - \phi)^2$ 
occurs at $\phi = \frac{c}{3}$, from which we define $a_c := \frac{4c^3}{27}$. For $a \in (-\infty,0) \cup (a_c,\infty)$, there exists only one critical (maximum) point of $U$, whereas for $a \in (0,a_c)$ there exist three critical points of $U$, two are local maximum and one is local minimum. In addition, $\phi = c$ is the pole singularity of $U(\phi)$ if $a \neq 0$. See Figure \ref{fig-1} for illustration of the three different cases of $U$.

It follows from dynamics of the Newton particle with the total energy in (\ref{first-order}) that all smooth mappings $x \mapsto \phi(x)$ for $a \in (-\infty,0] \cup [a_c,\infty)$ are unbounded. Although peaked and cusped periodic solutions exist in this case \cite{Len3,MP20}, we are only concerned with the smooth periodic solutions here. 

For $a \in (0,a_c)$, we shall label the critical points of $U$ 
as $\phi_1 < \phi_2 < \phi_3$. The following order is obtained from the graphical analysis of the cubic equation $a = \phi (c - \phi)^2$ on Fig. \ref{fig-1} (middle):
\begin{equation}
\label{ordering}
0 < \phi_1 < \frac{c}{3} < \phi_2 < c < \phi_3.
\end{equation}
The local minimum of $U$ at $\phi_2$ gives the center of the second-order equation (\ref{second-order}) at $(\phi_2,0)$. This implies that the smooth periodic solutions form a {\em period annulus}, that is, a punctured neighbourhood of the center $(\phi_2,0)$ enclosed by the homoclinic orbit connecting the saddle $(\phi_1,0)$. The phase portrait $(\phi,\phi')$ with the period annulus around the center $(\phi_2,0)$ is illustrated on Figure \ref{fig-plane}.

\begin{figure}[htb!]
	\includegraphics[width=0.32\textwidth]{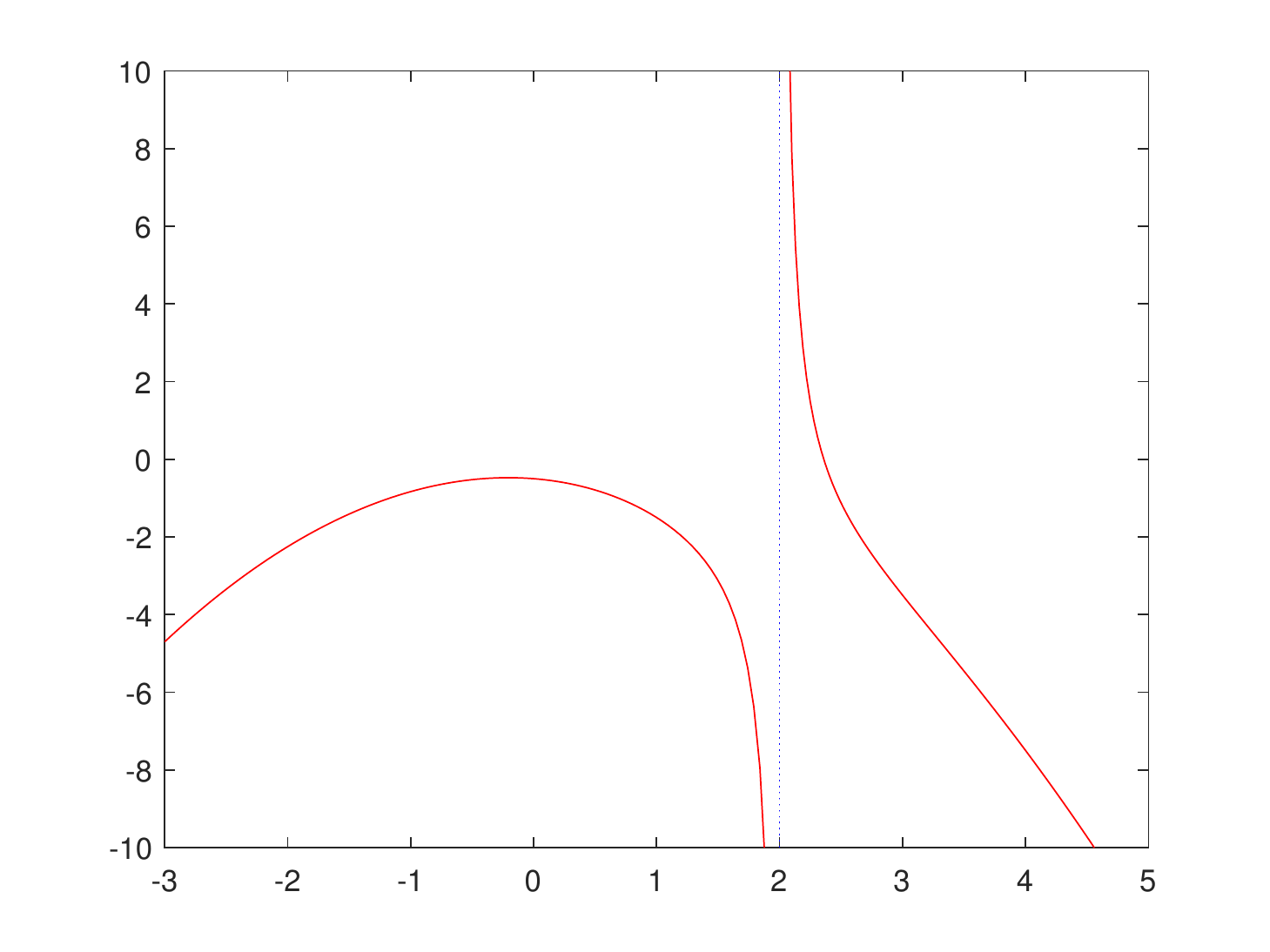}
	\includegraphics[width=0.32\textwidth]{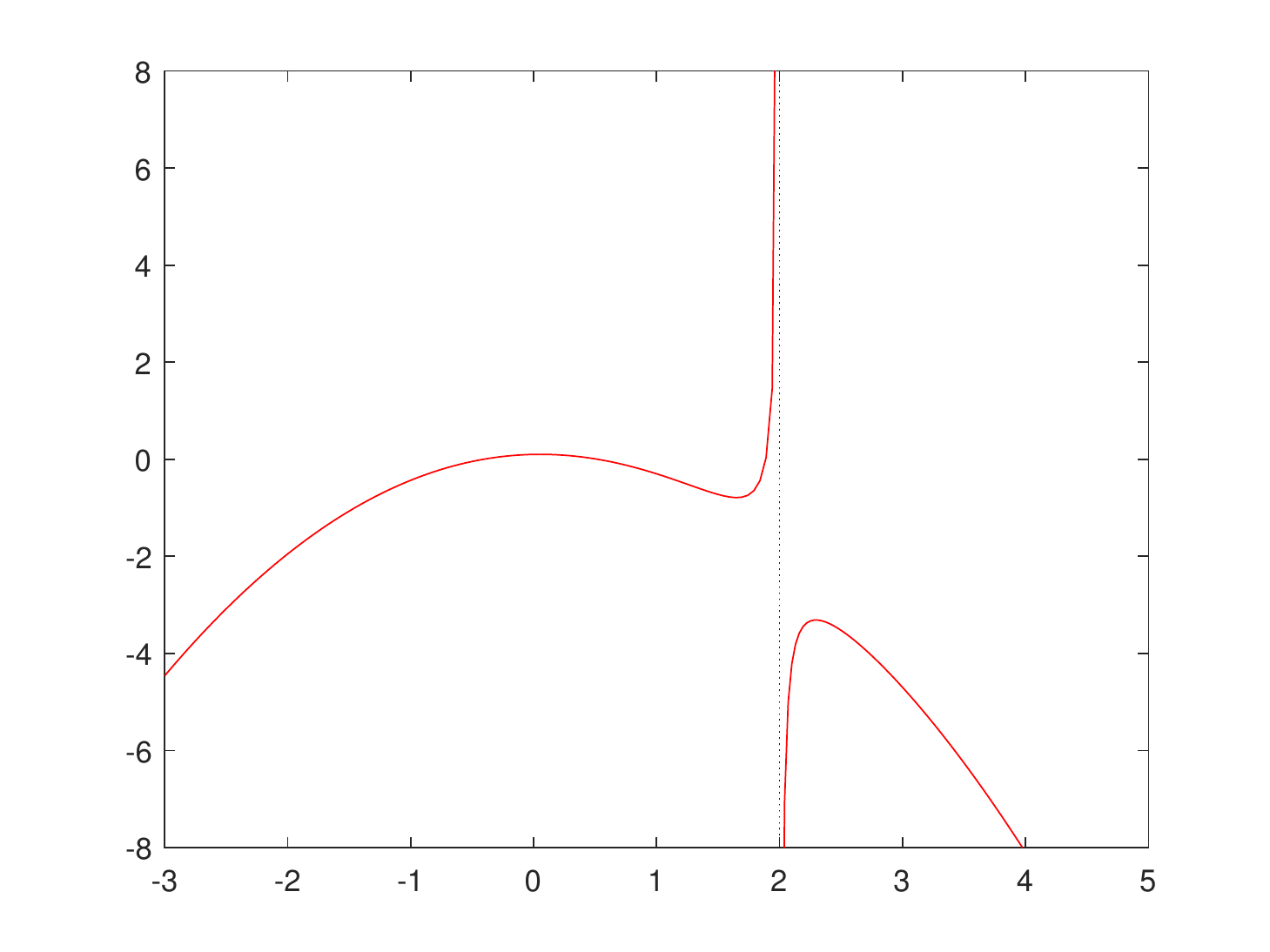}
	\includegraphics[width=0.32\textwidth]{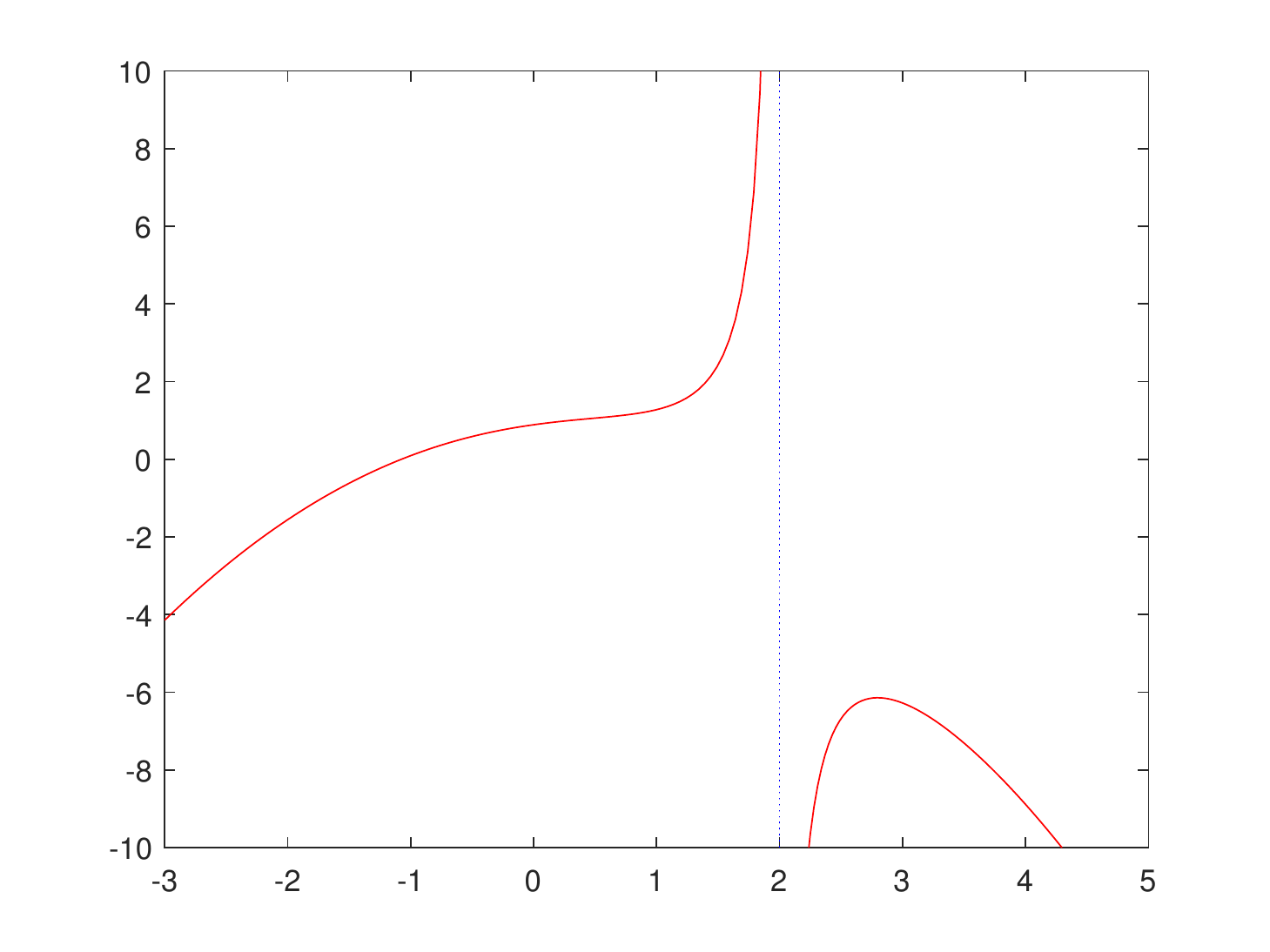}
	\caption{The graph of $U$ versus $\phi$ in (\ref{first-order}) for $a \in (-\infty,0)$ (left), $a \in (0,a_c)$ (middle), and $a \in (a_c,\infty)$ (right) for $c = 2$.} \label{fig-1}
\end{figure}

\begin{figure}[htb!]
	\includegraphics[width=0.6\textwidth]{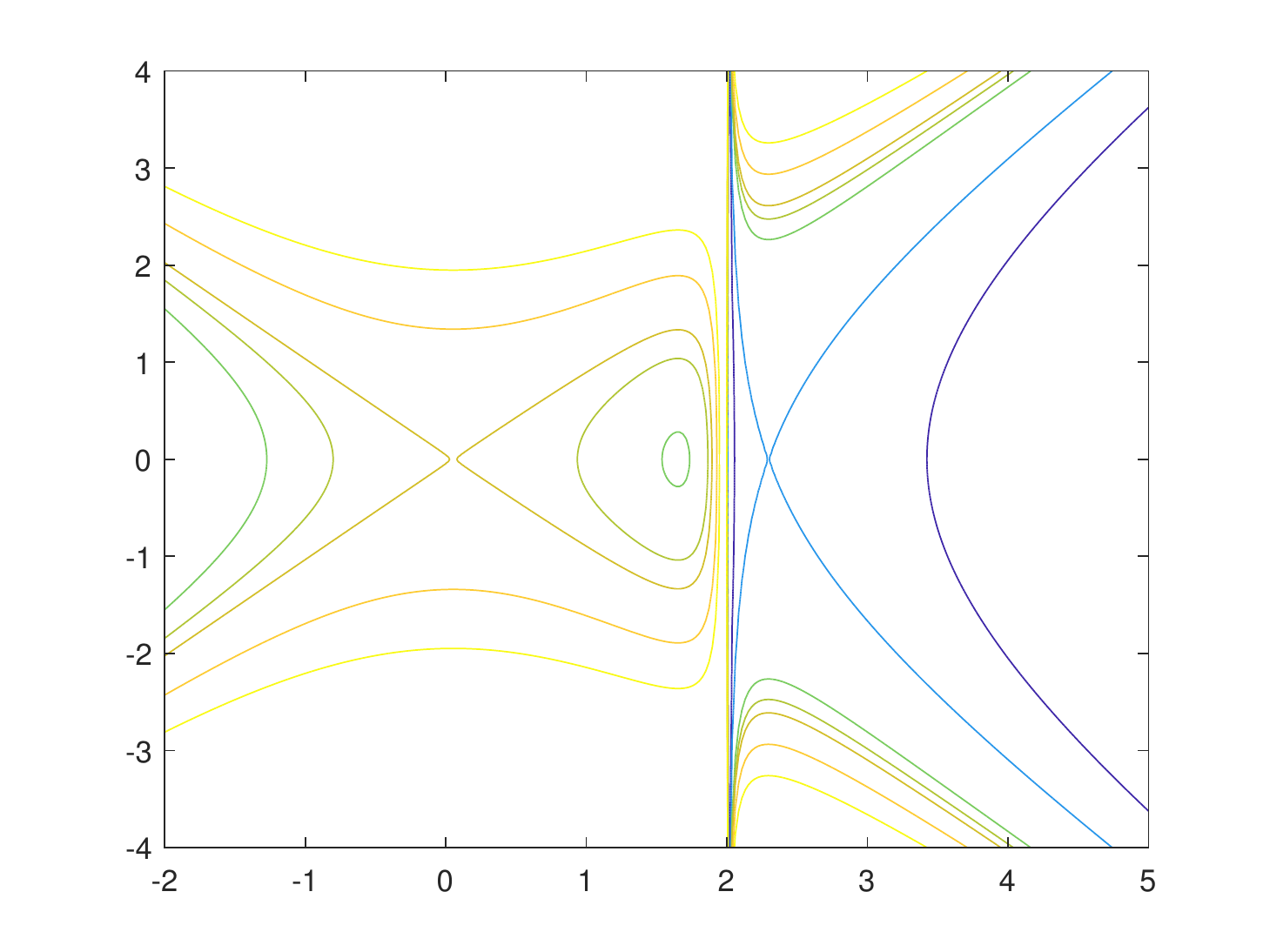}
	\caption{Phase portrait of the second-order equation 
		(\ref{second-order}) constructed from the level curves of the first-order invariant 
		(\ref{first-order}) for $a = 0.4$ and $c = 2$.} \label{fig-plane}
\end{figure}

The smooth periodic solutions for fixed $c > 0$ and $a \in (0,a_c)$ are parameterized by the parameter $b$ in $(b_-,b_+)$, where $b_- = U(\phi_2)$ and $b_+ = U(\phi_1)$. The following result summarizes the existence of smooth periodic solutions.

\begin{lemma}
	\label{period-to-energy}
	Fix $c > 0$. For a fixed $a \in (0,a_c)$ with $a_c := \frac{4c^3}{27}$, 
there exists a family of smooth periodic solutions $\phi$ of the second-order 
equation (\ref{second-order}) closed with the first-order invariant (\ref{first-order}) parameterized by $b \in (b_-,b_+)$, where 
$b_- = U(\phi_2)$ and $b_+ = U(\phi_1)$, such that $\phi \in (0,c)$. 
The solution $\phi$ is smooth with respect to parameters $a$, $b$, and $c$.
\end{lemma}

\begin{proof}
	Every periodic solution $\phi$ of the second-order equation (\ref{second-order}) corresponds to a periodic orbit of the planar system with the first integral given by (\ref{first-order}). Its level set parametrized by $b\in(b_-,b_+)$ defines the periodic orbits inside the period annulus around the center $(\phi_2,0)$, which exists if $a\in(0,a_c)$ as shown above. Due to the ordering (\ref{ordering}), the periodic solutions satisfy $\phi \in (\phi_1,c)$ which implies that $\phi \in (0,c)$. Since 
	the first-order invariant (\ref{first-order}) is smooth with respect to 
	parameters $a$, $b$, and $c$, the periodic orbits 
	inside the period annulus are also smooth with respect to parameters.
\end{proof}

Let us now define the period function $\mathfrak{L}(a,b,c)$ for the smooth $L$-periodic solutions of Lemma \ref{period-to-energy}. For fixed $a \in (0,a_c)$, 
$b \in (b_-,b_+)$, and $c > 0$, let $\phi_+$ and $\phi_-$ be turning points of the Newton's particle satisfying the ordering
\begin{equation}
\label{ordering-turning-points}
0 < \phi_1 < \phi_- < \phi_2 < \phi_+ < c < \phi_3.
\end{equation}
The turning points are roots of the algebraic equation 
\begin{equation}
\label{roots-turning-points}
(c - \phi_{\pm}) (2 b + \phi_{\pm}^2) = 2a.
\end{equation}
Without loss of generality, we place the maximum of $\phi$ at $x = 0$ and the minimum of $\phi$ at $x = \pm L/2$ so that $\phi(0) = \phi_+$ and $\phi(\pm L/2) = \phi_-$. Since the extremal values of $\phi$ are non-degenerate if $b \in (b_-,b_+)$, then $\phi''(0) < 0$ and $\phi''(\pm L/2) > 0$. It follows from (\ref{roots-turning-points}) that 
\begin{eqnarray}
\label{eq-1}
&&	(c \phi_{\pm} - \frac{3}{2} \phi_{\pm}^2 - b) \partial_a \phi_{\pm} = 1, \\
&&	(c \phi_{\pm} - \frac{3}{2} \phi_{\pm}^2 - b) \partial_b \phi_{\pm} = -(c-\phi_{\pm})
\label{eq-2}
\end{eqnarray}
and since $c - \phi_{\pm} > 0$ and
\begin{eqnarray}
\label{eq-3}
\phi''(0) = \frac{c \phi_+ -\frac{3}{2} \phi_+^2 - b}{c-\phi_+} < 0, \quad 
\phi''(\pm L/2) = \frac{c \phi_- -\frac{3}{2} \phi_-^2 - b}{c-\phi_-} > 0,
\end{eqnarray}
we have $\partial_a \phi_{\pm}, \partial_b \phi_{\pm} \neq 0$ with 
\begin{equation}
\label{opposite-signs}
{\rm sign}(\partial_a \phi_{\pm}) = -{\rm sign}(\partial_b \phi_{\pm}).
\end{equation}
The period function $\mathfrak{L}(a,b,c)$ is defined by integrating the quadrature (\ref{first-order})
\begin{equation}
\label{period-L-function}
\mathfrak{L}(a,b,c) := \int_{\phi_-}^{\phi_+} \frac{2\sqrt{c-\phi} d\phi}{\sqrt{(c-\phi)(2b + \phi^2) - 2a}}.
\end{equation}

Figure \ref{fig-domain} shows the existence region 
of the smooth periodic solutions on the $(a,b)$ plane 
for a fixed $c > 0$. Roots of the cubic equation $a = \phi (c-\phi)^2$ satisfying the ordering (\ref{ordering}) are computed 
numerically for every $a \in (0,a_c)$, from which we compute 
the values $b_- = U(\phi_2)$ and $b_+ = U(\phi_1)$. 
Plotting $b_-$ and $b_+$ versus $a$ gives the existence region 
enclosed by three boundaries. 

\begin{figure}[htb!]
	\includegraphics[width=0.6\textwidth]{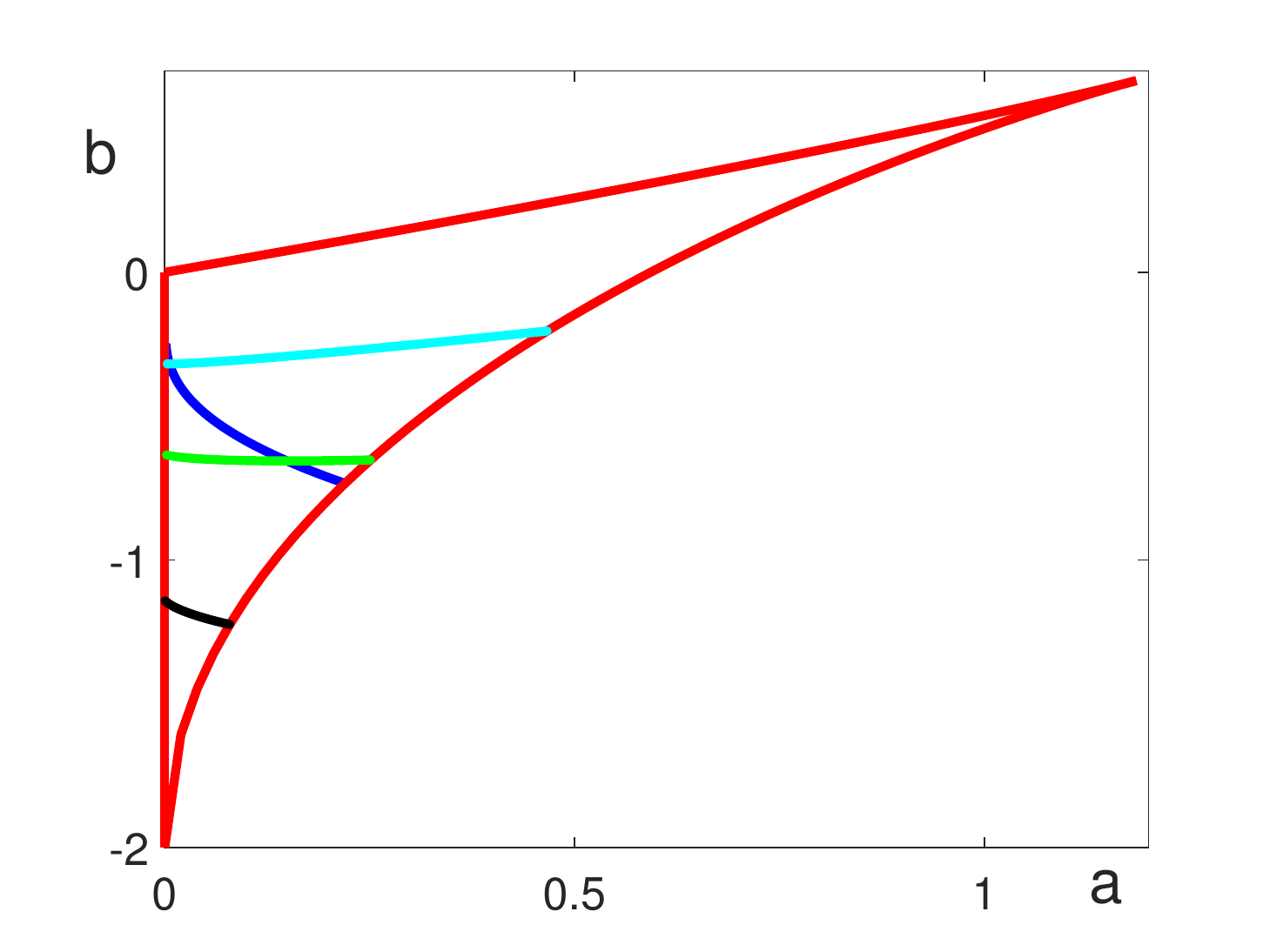}
	\caption{The existence region of smooth $L$-periodic solutions 
		on the parameter plane $(a,b)$ for $c = 2$ enclosed by three boundaries (red lines). The blue line 
		shows the values of $(a,b)$ for which the period function $\mathfrak{L}(a,b,c)$ has a maximum point in $a$ at fixed $(b,c)$.
		The black, green, and cyan lines show curves of fixed period 
		$L = \pi/2$, $L = 3\pi/4$, and $L = \pi$, respectively.} \label{fig-domain}
\end{figure}

The following three lemma clarify 
how the smooth periodic solutions transform when 
$(a,b)$ approach each boundary for a fixed $c > 0$.

\begin{lemma}
	\label{remark-right}
Fix $c > 0$ and $a \in (0,a_c)$. 
The smooth periodic solutions of Lemma \ref{period-to-energy}
transform as $b \to b_-(a)$ to the constant solutions. The limiting period function 
$$
\mathfrak{L}_-(a) := \mathfrak{L}(a,b_-(a),c)
$$ 
satisfies $\mathfrak{L}_-'(a) > 0$
with $\mathfrak{L}_-(a) \to 0$ as $a \to 0$ and $\mathfrak{L}_-(a) \to \infty$ as $a \to a_c$. The mapping $a \mapsto b_-(a)$ is $C^1$ and invertible with the inverse $a = a_-(b)$ for $b \in (-\frac{1}{2} c^2, \frac{1}{6} c^2)$.
\end{lemma}

\begin{proof}
	It follows from the ordering (\ref{ordering-turning-points}) that the boundary $b = b_-(a)$ corresponds to the center $\phi_- = \phi_+ = \phi_2$. Hence, $\phi(x) = \phi_2$ is constant in $x$. Linearization of the second-order equation (\ref{second-order}) at the center point $(\phi_2,0)$ determines the period $\mathfrak{L}_-(a) := \mathfrak{L}(a,b_-(a),c)$ in the form:
\begin{equation}
\label{period-boundary}
\mathfrak{L}_-(a) = \frac{2\pi}{\omega}, \quad \omega = \sqrt{\frac{2a}{(c-\phi_2)^3} - 1}.
\end{equation}
Along the curve $b = b_-(a)$, $a$ and $b$ can be parametrized by $\phi_2$ as
\begin{equation}
\label{param-boundary}
\left\{ \begin{array}{l} 
b = c \phi_2 - \frac{3}{2} \phi_2^2,\\
a = \phi_2 (c - \phi_2)^2,
\end{array} \right.
\end{equation}
which follow from equations (\ref{CHode}) and (\ref{second-order}) 
using that $\phi = \phi_2$ is constant.

Solving the first (quadratic) equation in (\ref{param-boundary}) for $\phi_2$ as 
$$
\phi_2 = \frac{c}{3} + \frac{\sqrt{c^2 - 6b}}{3}
$$
and substituting the second (cubic) equation  in (\ref{param-boundary}) for $a$ into (\ref{period-boundary}) yields 
$$
\omega^2 = \frac{2a}{(c-\phi_2)^3} - 1 =
\frac{3 \phi_2 - c}{c - \phi_2} = \frac{3 \sqrt{c^2 - 6b}}{2c - \sqrt{c^2 - 6b}}.
$$
This allows us to express $b$ explicitly in terms of $\mathfrak{L}_-(a)$ by 
\begin{equation}
\label{period-to-b-boundary}
b = \frac{c^2}{6} \left[ 1 - \frac{64 \pi^4}{(4 \pi^2 + 3 \mathfrak{L}_-(a)^2)^2} \right].
\end{equation}
It follows from (\ref{period-to-b-boundary}) that $\mathfrak{L}_-(a)$ increases in $b$ along the curve $b = b_-(a)$ and satisfies $\mathfrak{L}_-(a) \to 0$ as $b \to -\frac{1}{2} c^2$ 
(or equivalently, $a \to 0$) and $\mathfrak{L}_-(a) \to \infty$ as $b \to \frac{1}{6} c^2$ 
(or equivalently, $a \to a_c$). 
Since the parametrization (\ref{param-boundary}) implies that 
\begin{equation}
\label{argument-a-b}
\frac{db}{d\phi_2} = c - 3 \phi_2, \quad \frac{da}{d \phi_2} = (c - \phi_2) (c - 3 \phi_2) \quad \Rightarrow \quad \frac{da}{db} = c - \phi_2 
\end{equation}
and $\phi_2 < c$, the mapping $a \mapsto b_-(a)$ is $C^1$, invertible, and monotonically increasing from $(a,b) = (0,-\frac{1}{2}c^2)$ to $(a,b) = (\frac{4}{27}c^3, \frac{1}{6} c^2)$. Hence $\mathfrak{L}_-(a)$ is also increasing in $a$ along the curve $b = b_-(a)$. 
\end{proof}

\begin{lemma}
	\label{remark-upper}
	Fix $c > 0$ and $a \in (0,a_c)$. The smooth periodic solutions of Lemma \ref{period-to-energy} transform as $b \to b_+(a)$ to the solitary wave solutions with 
$$\mathfrak{L}_+(a) := \mathfrak{L}(a,b_+(a),c) = \infty.$$
	The mapping $a \mapsto b_+(a)$ is $C^1$ and invertible with the inverse $a = a_+(b)$ for $b \in (0, \frac{1}{6} c^2)$.
\end{lemma}

\begin{proof}
	It follows from ordering (\ref{ordering-turning-points}) that the boundary $b = b_+(a)$ corresponds to 
	$\phi_- = \phi_1$. Hence, $\phi(x)$ is the solitary wave solution satisfying $\phi(x) \to \phi_1$ as $x \to \pm \infty$ so that $\mathfrak{L}_+(a) := \mathfrak{L}(a,b_+(a),c) = \infty$. 
	Along the curve $b = b_+(a)$, $a$ and $b$ can be parametrized by $\phi_1$ as
	\begin{equation}
	\label{param-boundary-plus}
	\left\{ \begin{array}{l} 
	b = c \phi_1 - \frac{3}{2} \phi_1^2,\\
	a = \phi_1 (c - \phi_1)^2,
	\end{array} \right.
	\end{equation}
	which follow from equations (\ref{CHode}) and (\ref{second-order}) 
	using that $\phi = \phi_1$ is a constant solution if $b = U(\phi_1)$.
	By the same argument as in (\ref{argument-a-b}) but with $\phi_2$ replaced by $\phi_1$, the mapping $a \mapsto b_+(a)$ is $C^1$, invertible, and monotonically increasing from $(a,b) = (0,0)$ to $(a,b) = (\frac{4}{27}c^3, \frac{1}{6} c^2)$.
\end{proof}

\begin{lemma}
	\label{remark-left}
	Fix $c > 0$ and $b \in (-\frac{1}{2} c^2, 0)$.
The smooth periodic solutions of Lemma \ref{period-to-energy}
transform as $a \to 0$ to the peaked periodic solutions 
and the period function
$$
\mathfrak{L}_0(b) := \mathfrak{L}(0,b,c)
$$
satisfies $\mathfrak{L}_0'(b) > 0$ with $\mathfrak{L}_0(b) \to 0$ as $b \to -\frac{1}{2} c^2$ and $\mathfrak{L}_0(b) \to \infty$ as $b \to 0$.
\end{lemma}

\begin{proof}
If $a = 0$, then $\phi$ satisfies the equation $\phi'' - \phi = 0$ 
with 
$$
\max_{x \in [-\frac{L}{2},\frac{L}{2}]} \phi(x) = \phi(0) = c
$$ 
since $\phi_+ = c$ and $\phi_- = \sqrt{2|b|}$. This equation can be solved explicitly
\begin{equation}
\label{peaked-wave}
\phi(x) = c \frac{\cosh\left(\frac{L}{2}-|x|\right)}{\cosh\left(\frac{L}{2}\right)}, \quad x \in \left[-\frac{L}{2},\frac{L}{2}\right].
\end{equation}
The periodic wave is peaked at $x = 0$ and smooth at $x = \pm \frac{L}{2}$ with $\phi'\left(\pm \frac{L}{2}\right) = 0$. 
It follows from (\ref{quadra}) and (\ref{peaked-wave}) that 
\begin{equation}
\label{b-L}
b = \frac{1}{2} \left[ (\phi')^2 - \phi^2 \right] = -\frac{c^2}{2 \cosh^2\left(\frac{L}{2}\right)}, 
\end{equation}
in agreement with $\phi_- = \sqrt{2|b|}$. 
Hence $b \in (-\frac{1}{2}c^2,0)$ and it follows from (\ref{b-L}) that $L = \mathfrak{L}_0(b)$ increases 
in $b$ and satisfies $\mathfrak{L}_0(b) \to 0$ as $b \to-\frac{1}{2}c^2$ and $\mathfrak{L}_0(b) \to \infty$ as $b \to 0$. 
\end{proof}

\begin{remark}		
	The two boundaries of Lemmas \ref{remark-right} and \ref{remark-upper}
	intersect at $a = a_c = \frac{4c^3}{27}$, where the two critical points coallesce: $\phi_1 = \phi_2 = \frac{c}{3}$. This corresponds to $b = b_c = \frac{c^2}{6}$. 	The two boundaries intersect with the third boundary 
	$a = 0$ of Lemma \ref{remark-left} at $b = 0$ and $b = -\frac{1}{2} c^2$ respectively.
\end{remark}

Finally, we prove the main result of this section 
that the period function $\mathfrak{L}(a,b,c)$ is a strictly increasing function 
of $b$ for any fixed $a \in (0,a_c)$ and $c > 0$.

\begin{theorem}
	\label{theorem-increasing}
	Fix $c > 0$ and $a \in (0,a_c)$, where $a_c := \frac{4c^3}{27}$.
	The period function $\mathfrak{L}(a,b,c)$ is strictly increasing in $b$.
\end{theorem}

\begin{proof}
	Let $a = \phi_2 (c-\phi_2)^2$, where $\phi_2$ is the second root in 
	the ordering (\ref{ordering}). Using the transformation $\{x=\frac{\phi-\phi_2}{\phi_2}, y=\frac{v}{\phi_2}\}$, we can write the second-order equation (\ref{second-order}) as the planar system 
\begin{equation}\label{sys}
\left\{
\begin{array}{l}
x' =y,\\[2pt]
y' =1+x-\frac{\eta^2}{(\eta - x)^2},
\end{array}\right.
\end{equation}
associated with the Hamiltonian 
\begin{equation}
\label{potential-H}
H(x,y)=\frac{y^2}{2} + V(x), \quad V(x) := -\frac{x^2}{2} - x -\eta+ \frac{\eta^2}{\eta-x},
\end{equation}
where $\eta = \frac{c-\phi_2}{\phi_2}\in(0,2)$.

The potential $V$ is smooth away from the singular line $x=\eta$, has a local minimum at $x=0$ and two maxima at $x_1:=\eta-\frac{1}{2} -\frac{\sqrt{4\eta+1}}{2}$ and $x_3:=\eta-\frac{1}{2} +\frac{\sqrt{4\eta+1}}{2}$. 
The center at the origin is surrounded by periodic orbits $\gamma_h$, which lie inside the level curves $H(x,y) = h$ with $h\in(0,h^*)$ and $h^*=V(x_1)$. Denote by $x_2$ the unique solution of $V(x_1)=V(x)$ such that $x_1<0<x_2<\eta<x_3$, see Figure \ref{Fig_potential}. Finally, define the period function of the center $(0,0)$ of system (\ref{sys}) by 
\begin{equation}
\label{period-function-h}
\ell(h)=\int_{\gamma_h} \frac{dx}{y} \quad \text{ for each } h\in(0,h^*).
\end{equation}
Note that $b = \phi_2^2 \left(h + \eta - \frac{1}{2}\right)$ 
and $\mathfrak{L}(a,b,c) = \ell(h)$ for fixed $a \in (0,a_c)$ and $c > 0$.
Since $\phi_2$ is fixed, we have $\partial_b \mathfrak{L}(a,b,c) > 0$ if and only if $\ell'(h) > 0$.

\vspace{0.25cm}

\begin{figure}[htb!]
\includegraphics[width=0.5\textwidth]{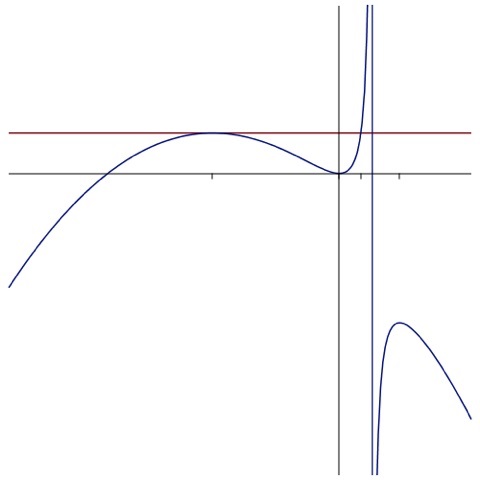}	
	\caption{The potential function $V(x)$ plotted for $\eta=\frac{1}{4}$.} 
	\label{Fig_potential}
\end{figure}

To prove that $\ell'(h) > 0$, we shall use a monotonicity criterion by Chicone \cite{Chic} for planar systems with Hamiltonians of the form (\ref{potential-H}), where $V$ is
a smooth function on $(x_1,x_2)$ with a nondegenerate relative minimum at the origin. The period function $\ell(h)$ is monotonically increasing in $h$ if the function 
$$
W(x) := \frac{V(x)}{(V'(x))^2}
$$
is convex in $(x_1,x_2)$. Hence, we have to prove that $W''(x) > 0$ for every $x \in (x_1,x_2)$. A straightforward computation shows that 
\begin{equation}
\label{second-der}
W''(x)=\frac{-3(\eta - x) R(x)}{(x-x_1)^4 (x-x_3)^4},
\end{equation}
where
\begin{equation*}
R(x)=\left( -2\,\eta+1 \right) {x}^{3}+\eta\, \left( 6\,\eta-7 \right) 
{x}^{2}-3\,{\eta}^{2} \left( 2\,\eta-3 \right) x+{\eta}^{2} \left( 
2\,\eta+1 \right)  \left( \eta-2 \right). 
\end{equation*}
Since $x_2 < \eta$, we need to show that $R(x) < 0$ for $x \in [x_1,x_2]$ and  $\eta \in (0,2)$. Note that $R(0)=\eta^2(2\eta+1)(\eta - 2) < 0$ for $\eta \in (0,2)$. 

The discriminant of $R$ with respect to $x$ is given by
\begin{equation}
\label{Discr}
\mbox{Disc}_x(R)=-4(4\eta + 1)(4\eta^2 - 16\eta + 27)\eta^4, 
\end{equation}
which is strictly negative for $\eta\in(0,2)$. Hence, for $\eta \neq \frac{1}{2}$ the cubic polynomial $R$ has exactly one real root, say $x_0$.

For $\eta < \frac{1}{2}$, it follows from the dominant behavior 
of $R$ that $R(x) \to -\infty$ as $x \to -\infty$.
Since $R(\eta) = -2 \eta^2 < 0$, it is clear that the only 
real root $x_0$ is located for $x_0 > \eta$. Therefore, 
$R(x) < 0$ for $x \in [x_1,x_2]$ with $x_2 < \eta$.

For $\eta > \frac{1}{2}$, we have $R(x) \to -\infty$ as $x \to +\infty$. We claim that 
\begin{equation}
\label{R-x-1}
R(x_1)= \frac{1}{2} ((\eta - 1)\sqrt{4\eta + 1} - \eta - 1) 
(4 \eta + 1) <0
\end{equation}
for  $\eta\in(0,2)$. Therefore, the only real root $x_0$ is located for $x_0 < x_1$ and $R(x) < 0$ for $x \in [x_1,x_2]$.
In order to prove (\ref{R-x-1}), we substitute $\eta=(w^2-1)/4$ into $R(x_1)$ and obtain $R(x_1)=\frac{1}{4} (w-3)(w+1)^2 w^2$ which is negative for $w\in(1,3)$. 

Finally, for $\eta = \frac{1}{2}$ we have that $R(x)=-2x^2 + 3x/2 - 3/4$ which is strictly negative for all $x$.

Hence $R(x) < 0$ for $x \in [x_1,x_2]$ if $\eta\in(0,2)$.
Therefore, $W''(x)>0$ for $x \in (x_1,x_2)$ and 
$\ell'(h) > 0$ follows by theorem proven in \cite{Chic}. 
\end{proof}

\begin{remark}
	The result of Theorem \ref{theorem-increasing}  can also be verified using the tools from \cite{ManVil09} and \cite{Vil2014}, where Hamiltonian systems with  Hamiltonian in the form $H(x,y) = \frac{1}{2} y^2 + V(x)$ are considered with $V(x) = \frac{1}{2m} x^{2m} + o(x^{2m})$, which is analytic in a neighborhood of $x=0$.
\end{remark}

\begin{remark}
	\label{remark-non-monotone}
	As claimed in Remark \ref{remark-existence}, the period function $\mathfrak{L}(a,b,c)$ has different monotonicity properties in $a$ for fixed $b \in (-\frac{1}{2} c^2,\frac{1}{6} c^2)$ and $c > 0$. To be precise, 
the period function 	
	\begin{itemize}
	\item is monotonically increasing in $a$ if $b \in (-\frac{1}{2} c^2,-(1 - \frac{\sqrt{2}}{\sqrt{3}}) c^2]$;
	\item has a single maximum point in $a$ if $b \in (-(1 - \frac{\sqrt{2}}{\sqrt{3}}) c^2,0)$; 
	\item is monotonically decreasing in $a$ if $b \in [0,\frac{1}{6} c^2)$.
	\end{itemize}
This result was obtained in \cite[Theorem 2.5]{GV}, where the second-order equation (\ref{CHode}) with the first-order invariant (\ref{quadra}) 
was reformulated into the system
	\begin{equation}
	\dot{x} = y, \quad \dot{y} = -\frac{y^2 + x - 3 x^2}{2(x + \nu)},
	\end{equation}
	where
	$$
	\nu := \frac{1}{6} \left[ \frac{2c}{\sqrt{c^2 - 6b}} - 1 \right].
	$$
	The value $b = -\frac{1}{2} c^2$ corresponds to $\nu = 0$, the value $b = -(1 - \frac{\sqrt{2}}{\sqrt{3}}) c^2$ corresponds to $\nu = -\frac{1}{10} + \frac{\sqrt{6}}{15}$, the value $b = 0$ 
	corresponds to $\nu = \frac{1}{6}$, and the value $b = \frac{1}{6} c^2$ corresponds to the limit $\nu \to \infty$.
\end{remark}

Figure \ref{fig-graphs-b} show the graphs of $\mathfrak{L}(a,b,c)$ versus $b$ for three cases of $a$. The period function is monotonically increasing in $b$ in agreement with Theorem \ref{theorem-increasing}. 
Figure \ref{fig-graphs} shows the graphs of $\mathfrak{L}(a,b,c)$ versus $a$ for three representative cases of $b$. The period function is increasing in $a$ 
for $b \in (-\frac{1}{2} c^2,-(1 - \frac{\sqrt{2}}{\sqrt{3}}) c^2]$ (left), 
has a single maximum point in $a$ if $b \in (-(1 - \frac{\sqrt{2}}{\sqrt{3}}) c^2,0)$ (middle) and is monotonically decreasing in $a$ if $b \in [0,\frac{1}{6} c^2)$ (right), in agreement with Remark \ref{remark-non-monotone}.

\begin{figure}[htb!]
	\includegraphics[width=0.32\textwidth]{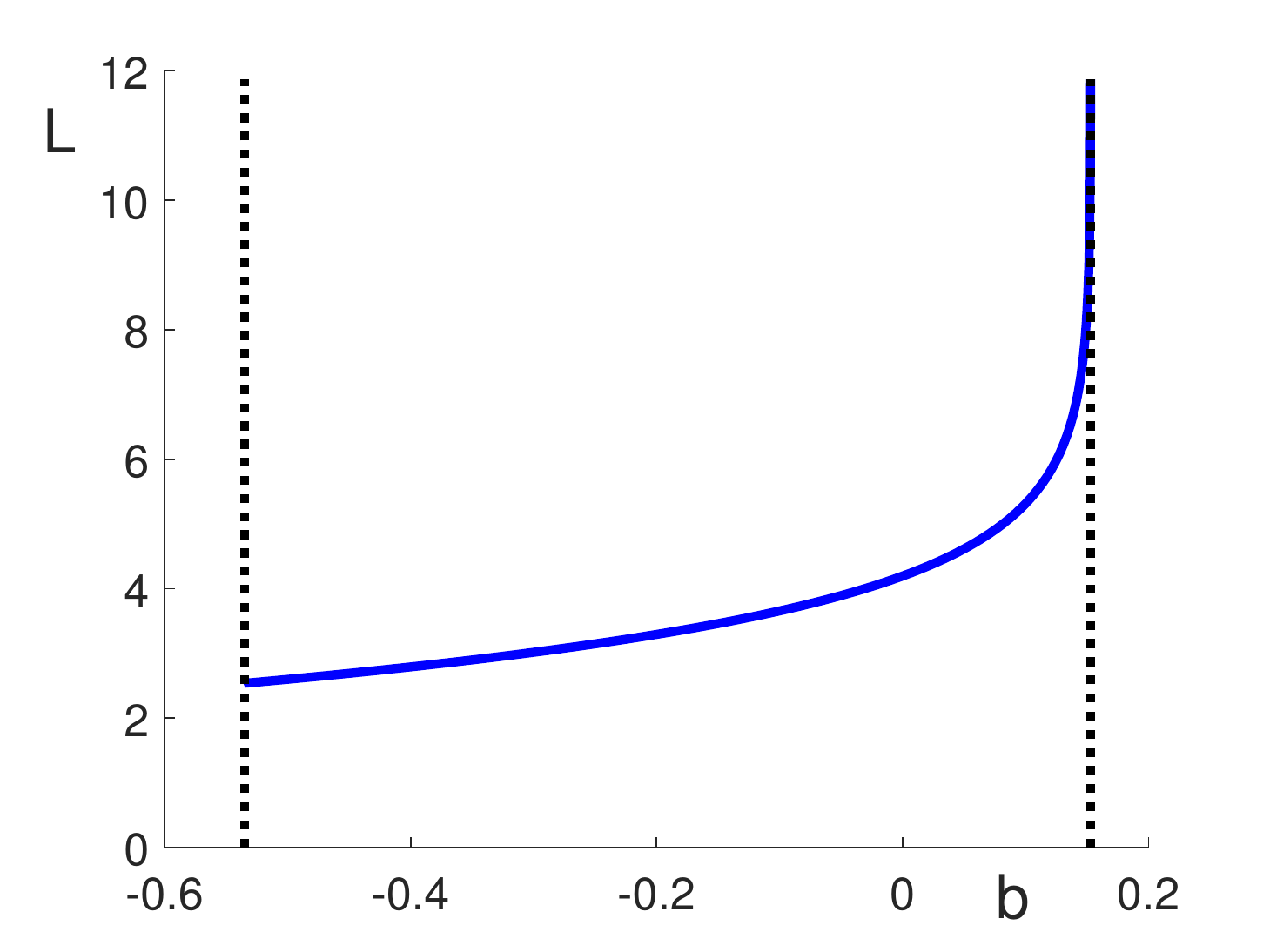}
	\includegraphics[width=0.32\textwidth]{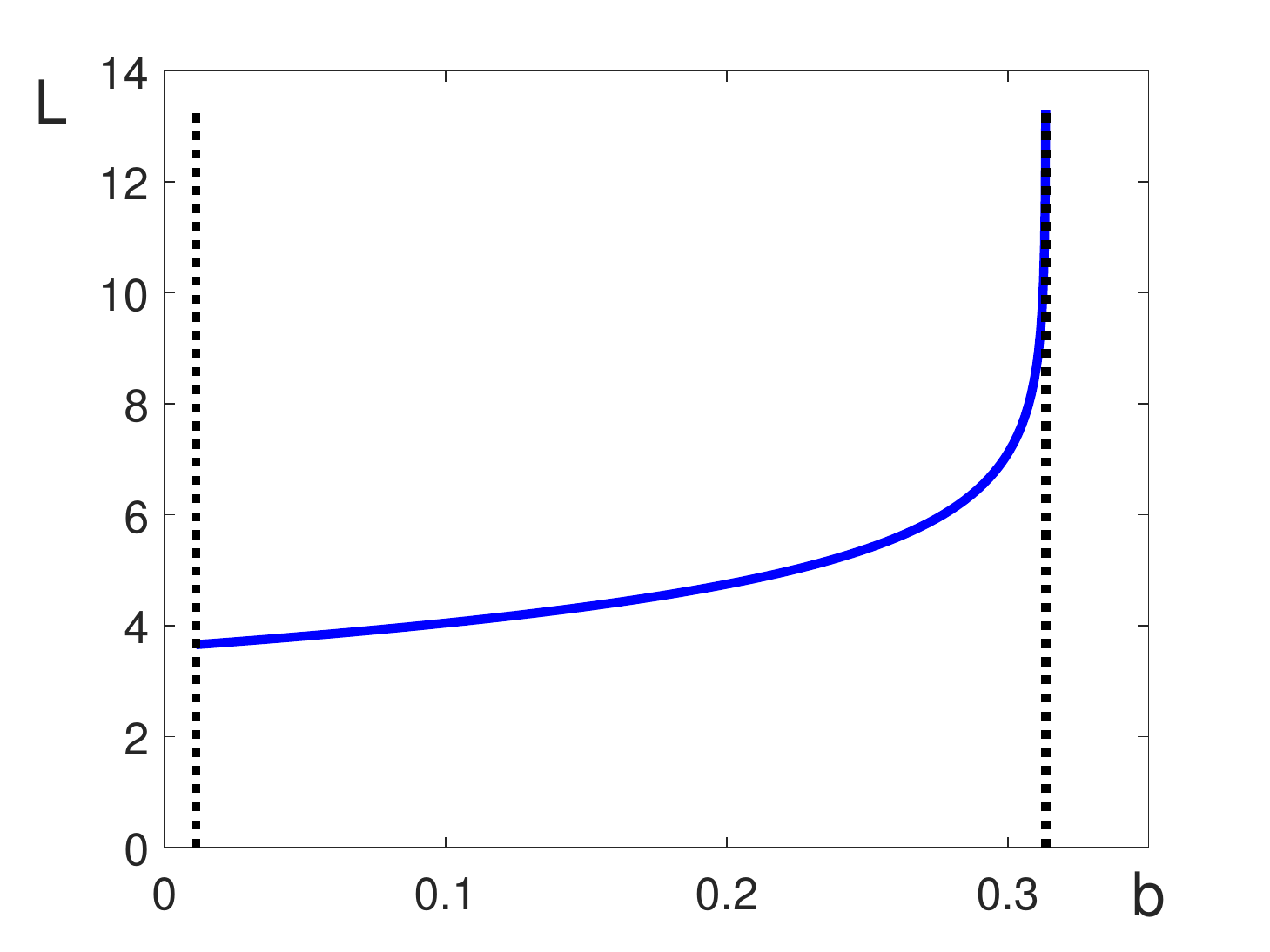}
	\includegraphics[width=0.32\textwidth]{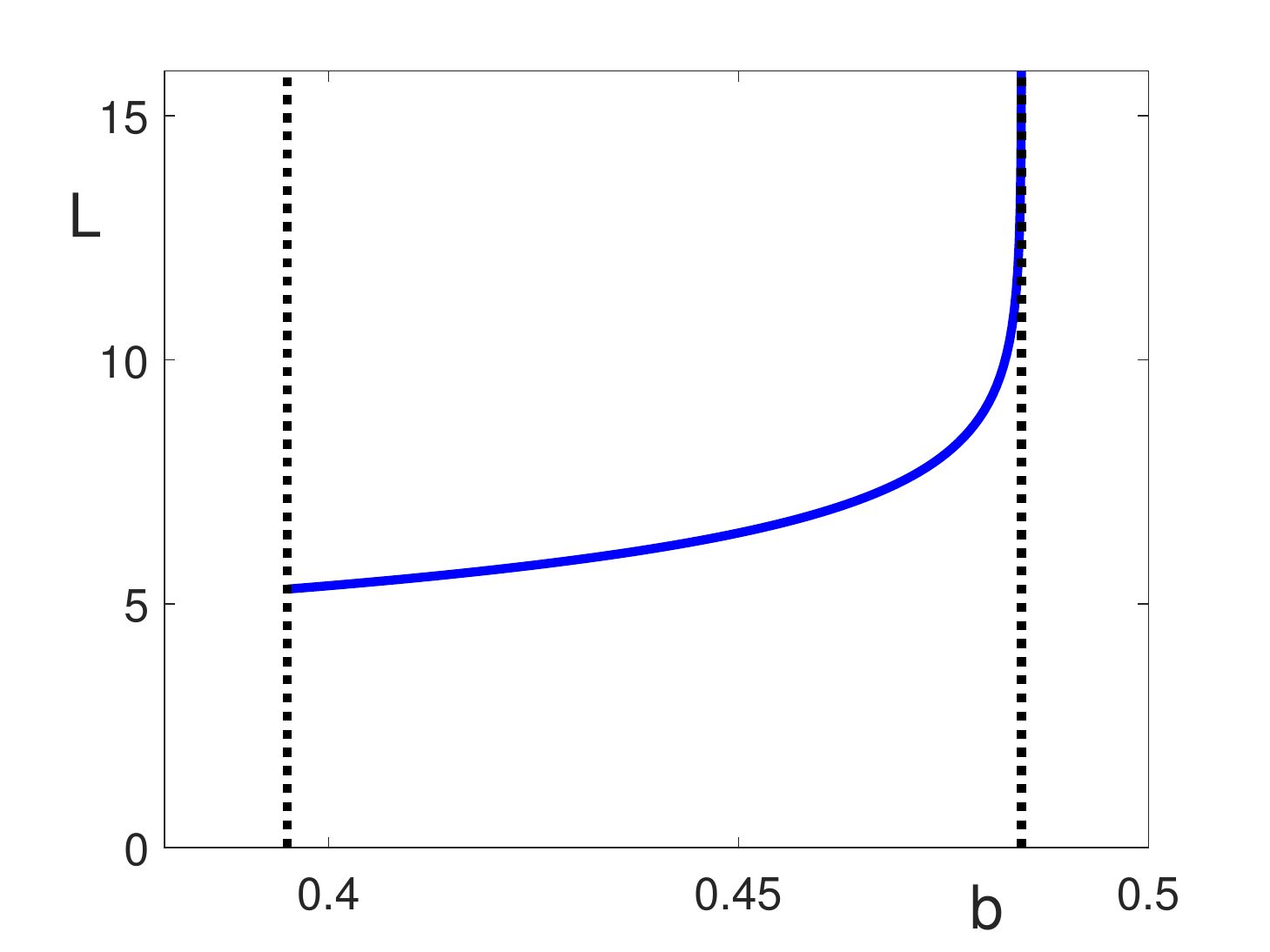}
	\caption{The period function $\mathfrak{L}(a,b,c)$ versus parameter $b$ for the smooth periodic solutions with $c = 2$ and three values of $a$: (left) $a = 0.3$, (middle) $a = 0.6$, and (right) $a = 0.9$.} \label{fig-graphs-b}
\end{figure}

\begin{figure}[htb!]
	\includegraphics[width=0.32\textwidth]{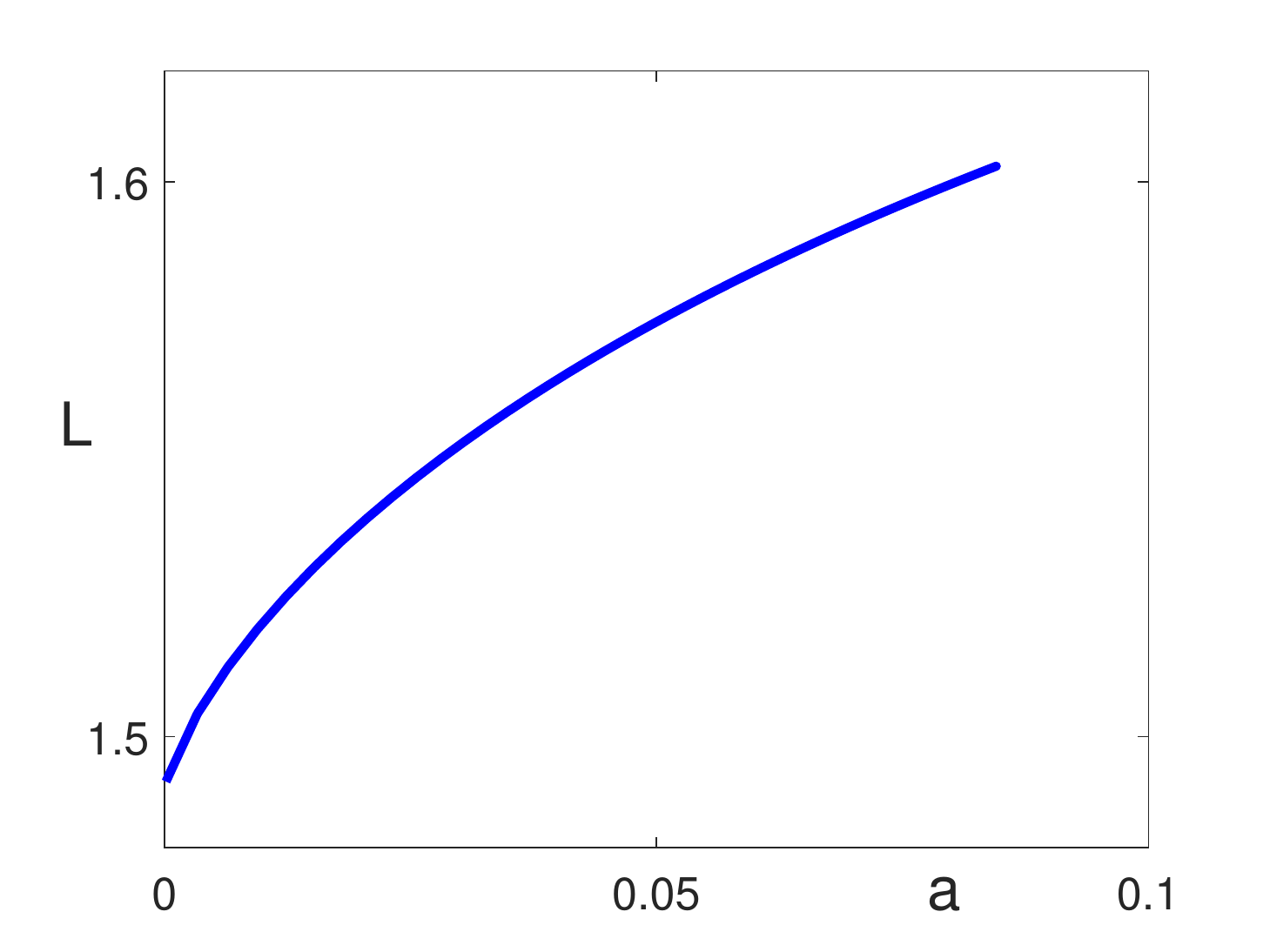}
	\includegraphics[width=0.32\textwidth]{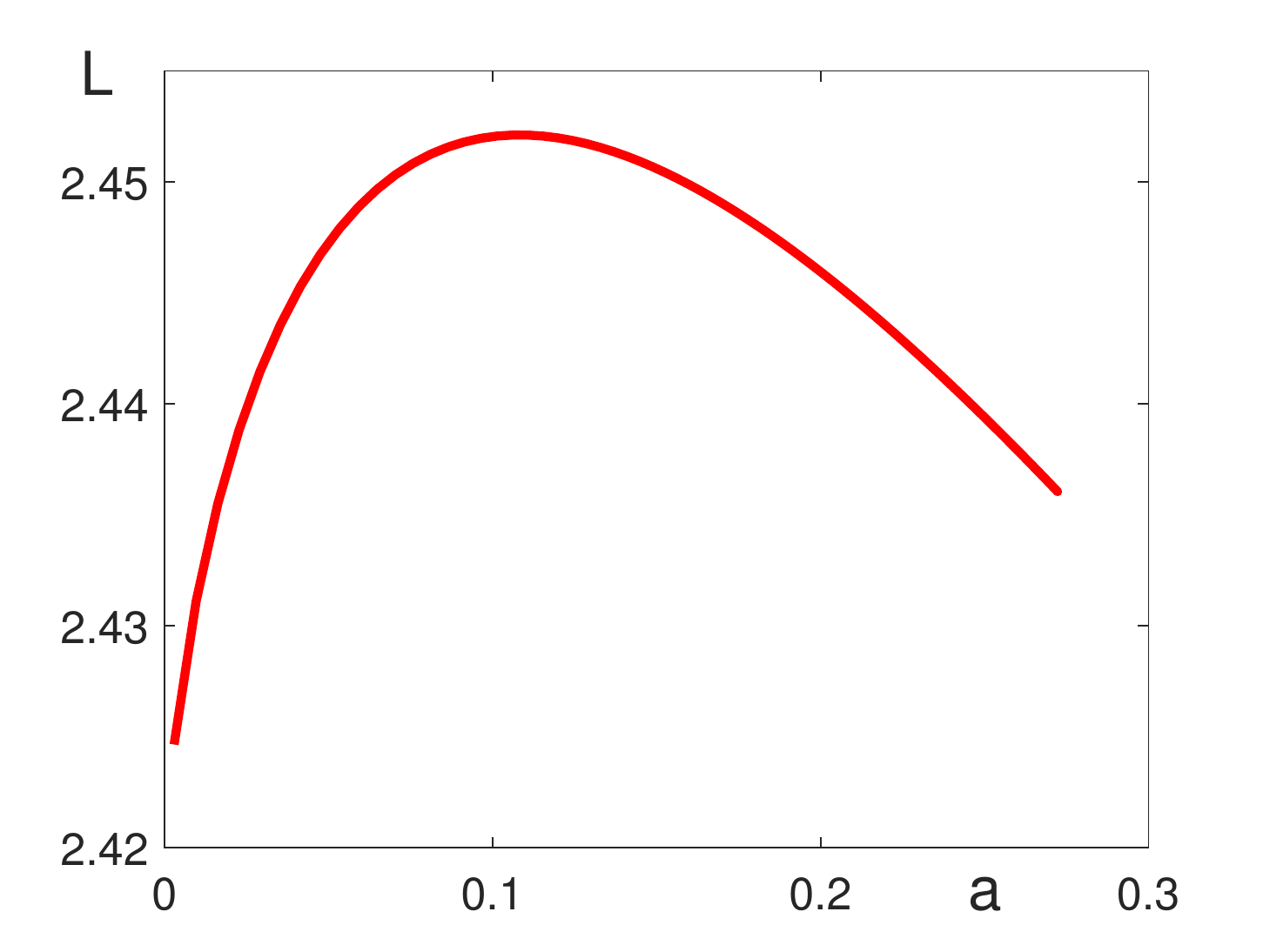}
	\includegraphics[width=0.32\textwidth]{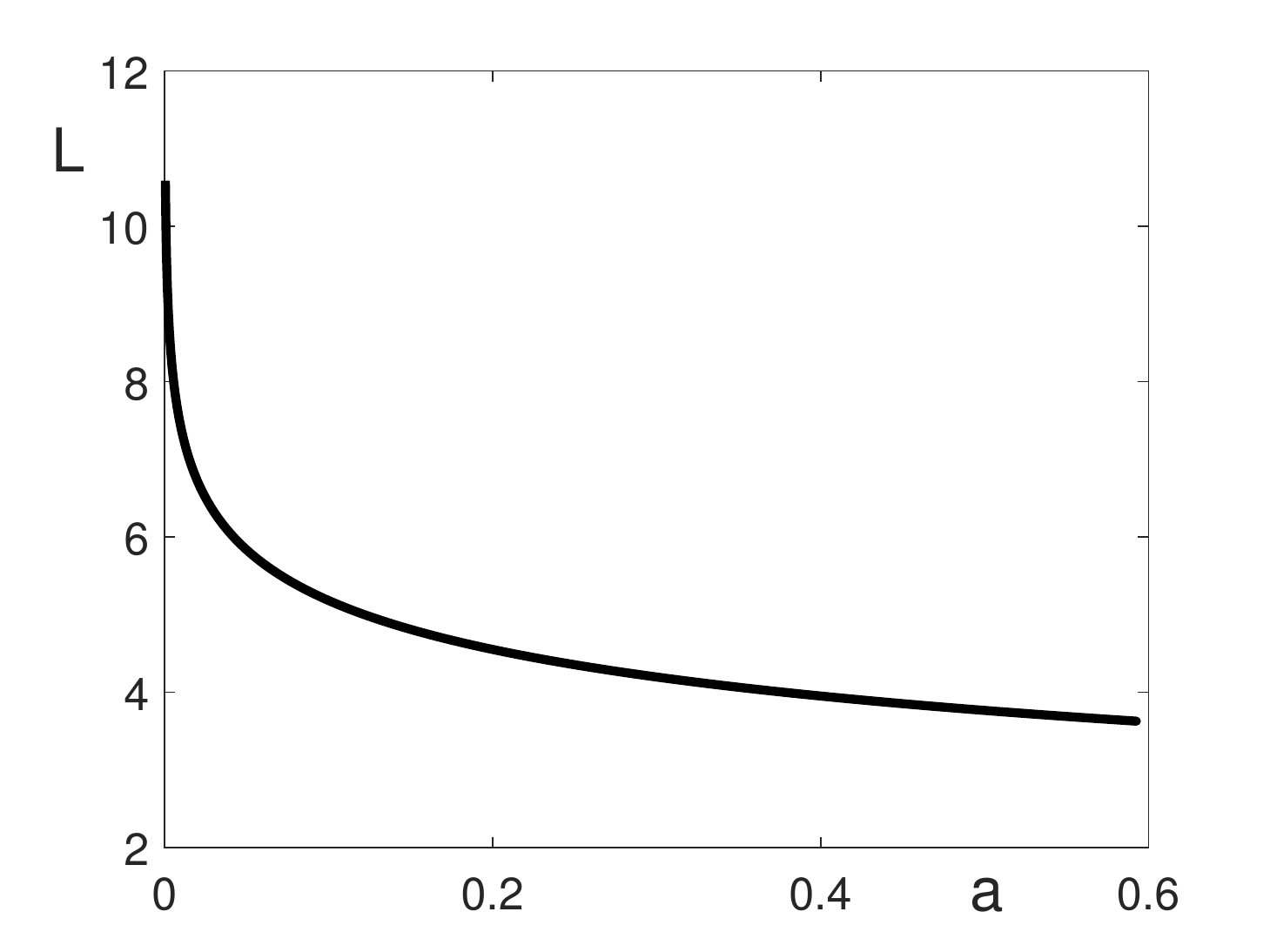}
	\caption{The period function $\mathfrak{L}(a,b,c)$ versus parameter $a$ for the smooth periodic solutions with $c = 2$ and three values of $b$: (left) $b = -1.2$, (middle) $b = -0.6$, and (right) $b = 0$.} \label{fig-graphs}
\end{figure}

For the study of spectral stability of periodic solutions in Section \ref{sec-4}, it is important to fix the period $L$ and consider the family of $L$-periodic solutions along a curve in the $(a,b)$ plane for a fixed $c > 0$. The following result provides this characterization of the $L$-periodic solutions. 

\begin{lemma}
	\label{lemma-fixed-period}
	Fix $c > 0$ and $L > 0$. There exists a $C^1$ mapping $a \mapsto b = \mathcal{B}_L(a)$ for $a \in (0,a_L)$ with some $a_L \in (0,\frac{4}{27}c^3)$
	and a $C^1$ mapping $a \mapsto \phi = \Phi_L(\cdot,a) \in H^{\infty}_{\rm per}$ of smooth $L$-periodic solutions along the curve $b = \mathcal{B}_L(a)$.
\end{lemma} 

\begin{proof}
	It follows from the monotonicity results in Lemmas \ref{remark-right} and \ref{remark-left} that 
	for every $c > 0$ and $L > 0$ there exists exactly one $L$-periodic solution on the left and right boundaries of the existence domain on the $(a,b)$-plane. The left boundary corresponds to $a = 0$ and the right boundary corresponds to $a = a_L$, where $a_L$ is uniquely defined 
	from the equation $\mathfrak{L}(a_L,b_-(a_L),c) = L$ with $b_-(a)$ defined in Lemma \ref{period-to-energy}.
	
	Since $\mathfrak{L}(a,b,c)$ is smooth in $(a,b,c)$ and it is strictly increasing in $b$ by Theorem \ref{theorem-increasing}, the existence of the $C^1$ mapping 
	$a \mapsto b = \mathcal{B}_L(a)$ for $a \in (0,a_L)$ follows by the implicit function theorem for $\mathfrak{L}(a,b,c) = L$ with fixed $c > 0$ and $L > 0$. Indeed, 
	$\partial_a \mathfrak{L} + \mathcal{B}_L'(a) \partial_b \mathfrak{L} = 0$ and since $\partial_b \mathfrak{L} > 0$, $\mathcal{B}_L'(a)$ is uniquely defined for every $a \in (0,a_L)$. 
	Since $\phi$ is smooth with respect to parameters by Lemma \ref{period-to-energy}, the mapping $a \mapsto \phi = \Phi_L(\cdot,a) \in H^{\infty}_{\rm per}$ is $C^1$ along the curve $b = \mathcal{B}_L(a)$.
\end{proof}

\begin{remark}
	\label{remark-fixed-period}
The mapping $b \mapsto \phi = \Psi_L(\cdot,b) \in H^{\infty}_{\rm per}$ may not be $C^1$ along the curve $b = \mathcal{B}_L(a)$ because of the non-monotonicity of $\mathfrak{L}(a,b,c)$ with respect to $a$. It follows from Remark \ref{remark-non-monotone} that there exists at most one point where $\mathcal{B}_L'(a) = 0$ and this is the minimum of the mapping $a \mapsto b = \mathcal{B}_L(a)$. The mapping $b \mapsto \phi = \Psi_L(\cdot,b) \in H^{\infty}_{\rm per}$ is not $C^1$ at the minimum point. 
\end{remark}

The blue curve in Fig. \ref{fig-domain} shows the location of the single maximum of the period function $\mathfrak{L}(a,b,c)$ in the $(a,b)$ plane for fixed $c = 2$
according to Remark \ref{remark-non-monotone}. The location 
of the maximum moves to the  right boundary as $b \to -(1 - \frac{\sqrt{2}}{\sqrt{3}})c^2$ and to the left boundary as $b \to 0$.

The black, green, and cyan curves in Figure \ref{fig-domain} also show 
the curves in the $(a,b)$ plane where smooth $L$-periodic solutions 
exist for three different values of $L$. The black curve with smaller $L$ 
does not intersect the blue curve and the family of $L$-periodic solutions remains smooth both in $a$ and $b$. However, the green and cyan curves with larger periods intersect the blue curve and the family of $L$-periodic solutions is smooth in $a$ but is not smooth in $b$ at the minimum point of $b = \mathcal{B}_L(a)$.

The numerical method used to generate Figures \ref{fig-domain}, \ref{fig-graphs-b}, and \ref{fig-graphs} is described in Appendix \ref{appendix}.

\section{Spectral properties of the linearized operator}
\label{sec-3}

Here we study the linearization of the CH equation (\ref{CH}) at the smooth periodic solutions of the system (\ref{CHode}), (\ref{second-order}), and (\ref{quadra}) and 
provide the proof of the last assertion of Theorem \ref{theorem-existence}. 

Adding a perturbation $v$ to the smooth travelling wave $\phi$ propagating with the same fixed speed $c$ in 
\begin{equation}
\label{perturbation-v}
u(x,t) = \phi(x-ct) + v(x-ct,t)
\end{equation}
gives the perturbation equation derived from the CH equation (\ref{CH}):
\begin{equation}
\label{CH-linearization}
(1 - \partial_x^2) (v_t - c v_x) + 3 \partial_x (\phi v) + 2 v v_x = \partial_x (\phi v_{xx} + \phi' v_x + \phi'' v) + 2 v_x v_{xx} + v v_{xxx}.
\end{equation}
Dropping the quadratic terms in $v$ yields the linearized evolution equation
\begin{equation}
\label{CH-lin}
v_t = - J \mathcal{L} v, 
\end{equation}
where $J$ is defined in (\ref{sympl-1}) and $\mathcal{L}$ 
is the linearized operator given by (\ref{hill}).

Recall that the travelling periodic wave is spectrally stable in the sense 
of Definition \ref{defstab-spectral} if the spectrum of the linearized 
operator $J \mathcal{L}$ in $L^2_{\rm per}$ is located on $i \mathbb{R}$. 
The following lemma reformulates the spectral stability criterion in terms of 
the linearized operator $\mathcal{K}$ introduced in (\ref{hill-m}).

\begin{lemma}
	\label{lem-equivalency}
	The spectrum of $J \mathcal{L}$ in $L^2_{\rm per}$ is located on the imaginary axis if and only if the spectrum of $(c-\phi)^{-1} \partial_x (c-\phi)^{-1}  \mathcal{K}$ in $L^2_{\rm per}$ is located on the imaginary axis. 
\end{lemma}

\begin{proof}
Consider the time evolution of the CH equation in the form (\ref{CHm}), where $m := u - u_{xx}$. We add a perturbation $p$ to the smooth travelling wave $\mu$ in
\begin{equation}
\label{perturbation-p}
m(x,t) = \mu(x-ct) + p(x-ct,t).
\end{equation}
It follows from the decompositions (\ref{perturbation-v}) and (\ref{perturbation-p}) that $\mu = \phi - \phi''$ and $p = v - v_{xx}$.  Substituting (\ref{perturbation-v}) and (\ref{perturbation-p})
into (\ref{CHm}) gives the perturbation equation
\begin{equation}
\label{CHm-linearization}
p_t + (\phi - c) p_x + 2p \phi' + v \mu' + 2 \mu v_x + v p_x + 2 p v_x = 0.
\end{equation}
Dropping the quadratic terms in $p$ and $v$ yields the linearized evolution equation
\begin{equation}
\label{CHm-lin}
p_t + (\phi - c) p_x + 2p \phi' + v \mu' + 2 \mu v_x = 0. 
\end{equation}
It follows from \eqref{second-order} that $(c-\phi)^2 \mu = a$ and hence
\begin{equation}
\label{CHm-relation}
v \mu' + 2 \mu v_x = \frac{2a \phi'}{(c-\phi)^3} v + \frac{2a}{(c-\phi)^2} v_x.
\end{equation}
Multiplying (\ref{CHm-lin}) by $c-\phi$ and using (\ref{CHm-relation})
yield the equivalent evolution form:
\begin{equation}
\label{CHm-lin-alter}
\frac{\partial}{\partial t} \left[ (c-\phi) p \right] = 
\frac{\partial}{\partial x} \left[ (c-\phi)^2 p  - \frac{2a}{c-\phi} v \right], 
\end{equation}
which can be written in the form 
\begin{equation}
\label{CHm-lin-final}
p_t = (c-\phi)^{-1} \partial_x (c-\phi)^{-1} \mathcal{K} p, 
\end{equation}
where $\mathcal{K}$ is the linearized operator given by (\ref{hill-m}).
It follows from the equivalence of (\ref{CH-lin}) and (\ref{CHm-lin-final}) under the transformation $v = (1-\partial_x^2)^{-1} p$ 
that $\lambda \in \sigma(J \mathcal{L})$ in $L^2_{\rm per}$ if and only if 
$\lambda \in \sigma[(c-\phi)^{-1} \partial_x (c-\phi)^{-1} \mathcal{K}]$ 
in $L^2_{\rm per}$, where $\sigma(A)$ denotes the spectrum of a linear operator $A$ in $L^2_{\rm per}$.
\end{proof}

In what follows we study the spectra of the linearized operators 
$\mathcal{L}$ and $\mathcal{K}$ in $L^2_{\rm per}$. The following lemma 
shows that the spectra of these operators are different. 

\begin{lemma}
	\label{lemma-L-K}
	The spectrum of $\mathcal{L}$ in $L^2_{\rm per}$ is purely discrete. 
	The spectrum of $\mathcal{K}$ in $L^2_{\rm per}$ consists 
	of the strictly positive continuous spectrum at 
	$$
	{\rm image}[(c-\phi)^3] = [(c-\phi_+)^3,(c-\phi_-)^3]
	$$ 
	and the discrete spectrum outside ${\rm image}[(c-\phi)^3]$, where $\phi_{\pm}$ are the turning points defined in (\ref{roots-turning-points}).
\end{lemma}

\begin{proof}
	Since $c - \phi > 0$ and $\phi \in H^{\infty}_{\rm per}$, 
	the linearized operator $\mathcal{L}$ with 
	the dense domain $H_{\rm per}^2 \subset L^2_{\rm per}$ 
	is a self-adjoint, unbounded operator in $L_{\rm per}^2$. 
	Consequently, $\sigma(\mathcal{L}) \subset \mathbb{R}$ is purely discrete 
	in $L^2_{\rm per}$ due to the compact embedding of $H^2_{\rm per}$ into $L^2_{\rm per}$.
	
	Since $c - \phi > 0$, the linearized operator $\mathcal{K}$ is 
	a self-adjoint, bounded operator in $L^2_{\rm per}$, which is the sum 
	of a bounded and a compact operator in $L^2_{\rm per}$. 
	Consequently, $\sigma(\mathcal{K}) \subset \mathbb{R}$ includes 
	both the continuous and discrete spectra in $L^2_{\rm per}$ denoted by $\sigma_c$ and $\sigma_d$ respectively. 
	Since the compact operator $-2a(1-\partial_x^2)^{-1}$ 
	is in the trace class in $L^2_{\rm per}$, Kato's
	theorem (Theorem 4.4 in \cite{Kato-text}) gives
	$$ 
	\sigma_c(\mathcal{K}) = \sigma_c((c-\phi)^3) = {\rm image}[(c-\phi)^3] = [(c-\phi_+)^3,(c-\phi_-)^3].
	$$
	Since $\phi_+ < c$, 
	$\sigma_c(\mathcal{K})$ is strictly positive.
\end{proof}

The following two theorems describe the non-positive part 
of the spectrum of $\mathcal{L}$ and $\mathcal{K}$ in $L^2_{\rm per}$. 
The proofs rely on Theorem 3.1 in \cite{neves} 
(see also the classical Floquet theory in \cite{eastham,Magnus}) 
and on Sylvester's inertial law theorem (see \cite[Theorem 2.2]{lopes}). 
These auxilary results are formulated in the following two propositions.

\begin{proposition}\cite{neves}
	\label{teo12}
	Let $\mathcal{M} := -\partial_x^2+Q(x)$
	be the Schr\"{o}dinger operator 
	with the even, $L-$periodic, smooth potential $Q$. Assume that 
	$\mathcal{M} w = 0$ is satisfied by a linear combination 
	of two solutions $\varphi_1$ and $\varphi_2$ satisfying
	$$
	\varphi_1(x+L) = \varphi_1(x) + \theta \varphi_2(x)
	$$ 
	and 
	$$
	\varphi_2(x+L) = \varphi_2(x)
	$$ 
	with some $\theta \in \mathbb{R}$. Assume that $\varphi_2$ has two zeros on the period of $Q$. The zero eigenvalue of $\mathcal{M}$ in $L^2_{\rm per}$ is simple if $\theta \neq 0$ and double if $\theta = 0$. It is the second eigenvalue of $\mathcal{M}$ if $\theta \geq 0$ and the third eigenvalue of $\mathcal{M}$ if $\theta < 0$. 
\end{proposition}

\begin{remark}
	Compared to \cite{neves}, we have interchanged the order of non-periodic $\varphi_1$ and periodic $\varphi_2$ so that our $\theta$ is negative relative to $\theta$ used in \cite{neves}.
\end{remark}

\begin{proposition}\cite{lopes}
	\label{prop-Lopes} Let $L$ be a self-adjoint operator in a Hilbert space $H$ and $S$ be a bounded invertible operator in $H$. Then, $S L S^*$ and $L$ have the same inertia, that is the dimension of the negative, null, and positive invariant subspaces of $H$.
\end{proposition}

We can now formulate and prove two theorems on the non-positive part 
of the spectrum of $\mathcal{L}$ and $\mathcal{K}$ in $L^2_{\rm per}$. 

\begin{theorem}
	\label{theolinop}
	The linearized operator $\mathcal{L} : H^2_{\rm per} \subset L^2_{\rm per} \to L^2_{\rm per}$ admits
	\begin{itemize}
		\item two negative eigenvalues and a simple zero eigenvalue  if $\partial_a \mathfrak{L} > 0$;
		\item one negative eigenvalue and a double zero eigenvalue if $\partial_a \mathfrak{L} = 0$;
		\item one negative eigenvalue and a simple zero eigenvalue if $\partial_a \mathfrak{L} < 0$,
	\end{itemize}
where $\mathfrak{L}(a,b,c)$ is the period function for the smooth 
periodic wave $\phi$ of Lemma \ref{period-to-energy}. The rest of the spectrum of $\mathcal{L}$ in $L^2_{\rm per}$ is strictly positive and bounded away from zero.
\end{theorem}

\begin{proof}
Due to the invariance of the CH equation (\ref{CH}) with respect to 
spatial translations, the third-order equation (\ref{third-order}) is equivalent to 
$\mathcal{L}\phi' = 0$, which means that $\phi' \in {\rm Ker}(\mathcal{L}) \subset H^2_{\rm per}$. 
On the other hand, differentiating of the second-order equation (\ref{CHode}) in $a$ is equivalent to 
$\mathcal{L} \partial_a \phi = 0$, which means that $\partial_a \phi$
is the second, linearly independent solution of $\mathcal{L}v = 0$.
Note that $\partial_a \phi$ is well-defined by Lemma \ref{period-to-energy} 
but may not be $L$-periodic in $x$. 

Let $\{ y_1,y_2\}$ be the fundamental set of solutions associated to the equation $\mathcal{L}v=0$ in $H^2(0,L)$ such that 
\begin{equation}
\label{fund-set}
\left\{ \begin{array}{l} y_1(0) = 1, \\
y_1'(0) = 0, \end{array} \right. \qquad 
\left\{ \begin{array}{l} y_2(0) = 0. \\
y_2'(0) = 1, \end{array} \right.
\end{equation}
As previously, we set $\phi(0) = \phi(L) = \phi_+$ 
for the smooth $L$-periodic solution of Lemma \ref{period-to-energy}, 
where $\phi_+$ is the turning point for the maximum of $\phi$ in $x$. 
Hence, we have $\phi'(0) = \phi'(L) = 0$ so that we define
\begin{equation}
\label{y1-y2}
y_1(x) := \frac{\partial_a \phi(x)}{\partial_a \phi_+}, \quad 
y_2(x) := \frac{\phi'(x)}{\phi''(0)},
\end{equation}
where $\partial_a \phi_+ \neq 0$ and $\phi''(0) \neq 0$ 
as follows from (\ref{eq-1}) and (\ref{eq-3}).
Differentiating of the boundary conditions $\phi(L) = \phi_+$ 
and $\phi'(L) = 0$ for $L = \mathfrak{L}(a,b,c)$ in $a$ 
yields $y_1(L) = y_1(0) = 1$ 
and 
$$
y_1'(L) = -\frac{\partial_a \mathfrak{L}}{\partial_a \phi_+} \phi''(0),
$$
which implies that 
\begin{equation}\label{relpq}
y_1(x+L) = y_1(x) + \theta y_2(x),
\end{equation}
where 
\begin{equation}
\label{thetaopera}
\theta = y_1'(L) = -\frac{\partial_a \mathfrak{L}}{\partial_a \phi_+} \phi''(0).
\end{equation}
Since $c - \phi_+> 0$, it follows from (\ref{eq-1}) and (\ref{eq-3}) that 
${\rm sign}(\theta) = -{\rm sign}(\partial_a \mathfrak{L})$.

In order to transform the spectral problem $\mathcal{L}v = \lambda v$ 
to the spectral problem $\mathcal{M} w = \lambda w$ for the Schr\"{o}dinger operator $\mathcal{M}$ in Proposition \ref{teo12}, we write $\mathcal{L} v = \lambda v$ as the second-order differential equation
	\begin{equation}
	\label{perspectprob}
	p(x)v''+q(x)v'+(r(x)+\lambda)v=0, 
	\end{equation}
with $p(x) := c-\phi(x)$, $q(x) := -\phi'(x)$, and $r(x) := -\phi''(x) + 3 \phi(x) - c$. The Liouville transformation
\begin{equation}\label{liouv}
D(x)=-\int_0^x\frac{\phi'(s)}{c-\phi(s)}ds=\ln\left(\frac{c-\phi(x)}{c-\phi(0)}\right)
\end{equation}
is nonsingular since $c - \phi > 0$.
Substituting the change of variables 
\begin{equation}\label{changvar}
v(x) = w(x) e^{-\frac{1}{2} D(x)} = w(x)\sqrt{\frac{c-\phi(0)}{c-\phi(x)}}.
\end{equation}
into the second-order equation (\ref{perspectprob}), we obtain the equivalent equation 
\begin{equation}
\label{hilleq}
-w''(x)+Q(x)w(x)=\lambda (c - \phi(x))^{-1} w(x),
\end{equation}
where 
$$
Q(x) := \frac{c-3\phi(x)}{c - \phi(x)} + \frac{\phi''(x)}{2(c - \phi(x))} +\frac{1}{4}\left(\frac{\phi'(x)}{c-\phi(x)}\right)^2. 
$$

With the transformation $w = (c-\phi)^{1/2} \hat{w}$, 
the spectral problem (\ref{hilleq}) is equivalent to the spectral 
problem for the operator $S \mathcal{M} S$, where $\mathcal{M} :=-\partial_x^2+Q(x)$ is self-adjoint in $L^2_{\rm per}$ 
and $S = (c - \phi)^{1/2}$ is a bounded and invertible multiplication 
operator in $L^2_{\rm per}$. By Proposition \ref{prop-Lopes}, 
the numbers of negative and zero eigenvalues of the spectral 
problem (\ref{hilleq}) coincides with those of the operator $\mathcal{M}$.

The operator $\mathcal{M}$ satisfies the condition of Proposition \ref{teo12} 
since $Q$ is even, $L-$periodic, and smooth. Since the set $\{ y_1,y_2\}$ is a fundamental set for the equation $\mathcal{L} v=0$ 
and the initial conditions $v(0) = w(0)$ and $v'(0) = w'(0)$ are preserved in the transformation (\ref{changvar}),  it follows that 
\begin{equation}
\label{rel-varphi}
\{ \varphi_1, \varphi_2 \} := \left\{\left(\frac{c-\phi(0)}{c-\phi}\right)^{-1/2} y_1,\left(\frac{c-\phi(0)}{c-\phi}\right)^{-1/2} y_2\right\}
\end{equation}
is the fundamental set of solutions to $\mathcal{M}w=0$. 
It follows from (\ref{relpq}) and (\ref{rel-varphi}) that  
\begin{equation}\label{relpq1}
\varphi_1(x+L) = \varphi_1(x) + \theta \varphi_2(x).
\end{equation}
where $\theta$ is given by the same expression (\ref{thetaopera}). 
Furthermore,  since $\phi'$ has two zeros in $\mathbb{T}_L$, 
the same is true for $y_2$ and $\varphi_2$. By the standard Floquet theory in \cite{eastham,Magnus}, it follows that $\lambda=0$ is the second or third eigenvalue of $\mathcal{M}$ in $L^2_{\rm per}$. If $\theta = 0$, then $\lambda = 0$ is the double eigenvalue so that it is the second eigenvalue of $\mathcal{M}$.  
If $\theta \neq 0$, then $\lambda = 0$ is a simple eigenvalue of $\mathcal{M}$.
By Proposition \ref{teo12}, it is the second eigenvalue if $\theta \geq 0$ 
and the third eigenvalue if $\theta < 0$. Due to the equivalence provided 
by the nonsingular transformation (\ref{changvar}), 
the same is true for the operator $\mathcal{L}$ in $L^2_{\rm per}$, 
which yields the assertion of the theorem since
${\rm sign}(\theta) = -{\rm sign}(\partial_a \mathfrak{L})$.
\end{proof}

\begin{theorem}
	\label{theorem-K}
	The linearized operator $\mathcal{K} : L^2_{\rm per} \to L^2_{\rm per}$ admits
	\begin{itemize}
		\item two negative eigenvalues and a simple zero eigenvalue  if $\partial_b \mathfrak{L} < 0$;
		\item one negative eigenvalue and a double zero eigenvalue if $\partial_b \mathfrak{L} = 0$;
		\item one negative eigenvalue and a simple zero eigenvalue if $\partial_b \mathfrak{L} > 0$,
	\end{itemize}
where $\mathfrak{L}(a,b,c)$ is the period function for the smooth periodic wave $\phi$ of Lemma \ref{period-to-energy}. The rest of the spectrum of $\mathcal{K}$ is strictly positive and bounded away from zero.
\end{theorem}

\begin{proof}
The linear operator $\mathcal{K}$ is congruent to another operator 
$\mathcal{K}_0$ by the transformation
	\begin{equation}
	\label{opKK0}
	\mathcal{K} = 
	(1-\partial_x^2)^{-1/2} \mathcal{K}_0 (1-\partial_x^2)^{-1/2},
	\end{equation}
	where 
	\begin{equation}
	\label{oper-K}
	\mathcal{K}_0 = (1 - \partial_x^2)^{1/2}(c-\phi)^3(1-\partial_x^2)^{1/2} - 2a.
	\end{equation}
	Since $S := (1-\partial_x^2)^{-1/2}$ is a bounded and invertible operator in $L^2_{\rm per}$ and $\mathcal{K}_0$ is self-adjoint in $L^2_{\rm per}$, it follows by Proposition \ref{prop-Lopes} that $\mathcal{K}$ and $\mathcal{K}_0$ in (\ref{opKK0}) have the same inertia, that is, the dimension of the negative, null, and positive invariant subspaces of $L^2_{\rm per}$. By Lemma \ref{lemma-L-K}, the positive invariant subspace of $\mathcal{K}$ is infinite-dimensional. Hence, we study the non-positive spectrum of $\mathcal{K}_0$.
		
	It follows from (\ref{oper-K}) that $\mathcal{K}_0$ is an unbounded self-adjoint operator defined in $L_{\rm per}^2$ with densely defined domain $H_{\rm per}^2 \subset L^2_{\rm per}$. The spectrum of $\sigma(\mathcal{K}_0)$ is given by the union of the continuous and discrete spectra. However, since the embedding of $H_{\rm per}^2$ into $L_{\rm per}^2$ is compact, the continuous spectrum is an empty set.	
Hence, we consider the spectral problem for the discrete spectrum:
	\begin{equation}
	\label{specK0}
	\mathcal{K}_0 w =\lambda w, \quad w \in H^2_{\rm per},
	\end{equation}
	where $\lambda\in \mathbb{R}$ is an isolated eigenvalue of $\mathcal{K}_0$ and $w \not\equiv0$ is the corresponding eigenfunction. Considering the change of variables $w := (1-\partial_x^2)^{1/2}v$, it follows from (\ref{specK0}) that
	\begin{equation}
	\label{specK0-eq1}
	(1-\partial_x^2)^{1/2} \left[(c-\phi)^3(1-\partial_x^2)v-(2a+\lambda)v\right]=0.
	\end{equation}
	Since $(1-\partial_x^2)^{1/2}$ is invertible in $L^2_{\rm per}$, the spectral problem (\ref{specK0-eq1}) is equivalent to the spectral problem 
		\begin{equation}
	\label{specK0-eq3}
	\mathcal{M} v = \lambda (c-\phi)^{-3} v,
	\end{equation}
	where $\mathcal{M}$ is the Schr\"{o}dinger operator given by 
	\begin{equation}
	\label{operator-K-2}
	\mathcal{M} := -\partial_x^2 + 1 - \frac{2a}{(c-\phi)^3}.
	\end{equation}
	With the transformation $v = (c-\phi)^{3/2} \hat{v}$, 
	the spectral problem (\ref{specK0-eq3}) is equivalent to that 
	for the operator $S \mathcal{M} S$, where $S := (c-\phi)^{3/2}$ is a bounded and invertible operator in $L^2_{\rm per}$ and 
	$\mathcal{M}$ is a self-adjoint operator in $L^2_{\rm per}$. 
	By Proposition \ref{prop-Lopes}, operators $\mathcal{M}$ and $S \mathcal{M} S$ have the same inertia in $L^2_{\rm per}$. 
	
	Finally, we study the non-positive spectrum of $\mathcal{M}$.
	It follows from the differential equation (\ref{second-order}) that 
$$
	\mathcal{M} \phi'= 0, \quad \mathcal{M} \partial_b \phi = 0.
	$$
	Therefore, the general solution of $\mathcal{M} v = 0$ is 
	given by a linear combination of two linearly independent solutions 
	\begin{equation}
	\label{y1-y2-again}
	y_1(x) := \frac{\partial_b \phi(x)}{\partial_b \phi_+}, \quad 
	y_2(x) := \frac{\phi'(x)}{\phi''(0)},
	\end{equation}
	where $\partial_b \phi_+ \neq 0$ and $\phi''(0) \neq 0$ 
	as follows from (\ref{eq-1}) and (\ref{eq-3}).
	Differentiating of the boundary conditions $\phi(L) = \phi_+$, 
	and $\phi'(L) = 0$ for $L = \mathfrak{L}(a,b,c)$ in $b$ 
	yields $y_1(L) = y_1(0) = 1$ 
	and 
	$$
	y_1'(L) = -\frac{\partial_b \mathfrak{L}}{\partial_b \phi_+} \phi''(0),
	$$
	so that 
	\begin{equation}\label{relpq-again}
	y_1(x+L) = y_1(x) + \theta y_2(x), \quad 
	\theta := y_1'(L) = -\frac{\partial_b \mathfrak{L}}{\partial_b \phi_+} \phi''(0).
	\end{equation}
	Since $c - \phi_+> 0$, 
	$\phi''(0) < 0$ and $\phi''(0) \partial_b \phi_+= -1$, as follows from (\ref{eq-1}) and (\ref{eq-3}), we obtain
	${\rm sign}(\theta) = {\rm sign}(\partial_b \mathfrak{L})$. 
	The assertion of the theorem follows by Proposition \ref{teo12} 
	due to equivalence of the negative and null subspaces of $\mathcal{K}_0$ and $\mathcal{M}$ and the inertial law between $\mathcal{K}$ and $\mathcal{K}_0$ and between $\mathcal{M}$ and $S \mathcal{M} S$.
\end{proof}

\begin{remark}
By Remark \ref{remark-non-monotone}, we have $\partial_a \mathfrak{L} > 0$ 
for every point $(a,b)$ below the blue curve  in the existence region of Fig. \ref{fig-domain} and $\partial_a \mathfrak{L} < 0$ for every point $(a,b)$ above 
the blue curve. Therefore, the count of negative eigenvalues of 
the linearized operator $\mathcal{L}$ in Theorem \ref{theolinop} 
changes depending on the point $(a,b)$.  
However, by Theorem \ref{theorem-increasing}, $\partial_b \mathfrak{L} > 0$ 
for every point $(a,b)$ inside the existence region, hence the linearized operator $\mathcal{K}$ in Theorem \ref{theorem-K} admits a simple negative eigenvalue and a simple zero eigenvalue for every $(a,b)$ in the existence region.
\end{remark}

\section{Spectral stability of periodic waves}
\label{sec-4}

Here we study the linearized CH equations (\ref{CH-lin}) 
and (\ref{CHm-lin-final}) and prove the spectral stability 
of periodic waves stated in Theorem \ref{theorem-stability}. We start by deducing 
the constraints on the perturbations $v \in H^2_{\rm per}$ 
and $p \in L^2_{\rm per}$ satisfying these linearized equations. 

\begin{lemma}
	\label{lem-stab-1}
	Let $v_0 \in X_0 \subset H^2_{\rm per}$, where $X_0$ is given by 
	\begin{equation}
	\label{orth-cond}
	X_0 := \left\{ v \in H^2_{\rm per} : \quad 
	\langle 1, v \rangle = 0, \quad \langle \phi - \phi'', v \rangle = 0 \right\}.
	\end{equation}
	If $v \in C^0(\mathbb{R},H^2_{\rm per}) \cap C^1(\mathbb{R},H^1_{\rm per})$ is a solution to the linearized CH equation (\ref{CH-lin}) 
with initial data $v_0$, then $v(t,\cdot) \in X_0 \subset H^2_{\rm per}$ for all $t \in \mathbb{R}$.
\end{lemma}

\begin{proof}
	Conservation of the two orthogonality conditions in $X_0$
	in the time evolution of the linearized CH equation (\ref{CH-lin}) 
	are checked directly using integration by parts:
	$$
	\frac{d}{dt} \langle 1, v \rangle = 
	\langle 1, \partial_x (1-\partial_x^2)^{-1} \mathcal{L} v \rangle = 0
	$$
	and
	$$
	\frac{d}{dt} \langle \phi - \phi'', v \rangle = 
\langle \phi - \phi'', (1-\partial_x^2)^{-1} \partial_x \mathcal{L} v \rangle 
= \langle \phi', \mathcal{L} v \rangle = \langle \mathcal{L} \phi', v \rangle = 0,
$$
where we recall that $\mathcal{L} \phi' = 0$. Integrations by parts are justified since $\phi \in H^{\infty}_{\rm per}$ and $v(t,\cdot) \in H^2_{\rm per}$ is in the domain of $\mathcal{L}$.
\end{proof}

\begin{corollary}
	\label{cor-stab-1}
	Let $p_0 \in Y_0 \subset L^2_{\rm per}$, where $Y_0$ is given by 
	\begin{equation}
	\label{orth-cond-again}
	Y_0 := \left\{ p \in L^2_{\rm per} : \quad 
		\langle 1, p \rangle = 0, \quad \langle \phi, p \rangle = 0 \right\}.
	\end{equation}
	If $p \in C^0(\mathbb{R},L^2_{\rm per}) \cap C^1(\mathbb{R},H^{-1}_{\rm per})$ is a solution to the linearized CH equation (\ref{CHm-lin-final}) 
	with initial data $p_0$, then $p(t,\cdot) \in Y_0 \subset H^2_{\rm per}$ for all $t \in \mathbb{R}$.
\end{corollary}

\begin{proof}
Orthogonality conditions in (\ref{orth-cond-again}) follow from
those in  (\ref{orth-cond}) by using the relation $v = (1-\partial_x^2)^{-1} p$ 
between solution $v \in C^0(\mathbb{R},H^2_{\rm per}) \cap C^1(\mathbb{R},H^1_{\rm per})$ of (\ref{CH-lin})  
and the corresponding solution $p \in C^0(\mathbb{R},L^2_{\rm per}) \cap C^1(\mathbb{R},H^{-1}_{\rm per})$ of (\ref{CHm-lin-final}). 
In particular, 
$$
\langle 1, v \rangle = \langle 1, v - v_{xx} \rangle = \langle 1, p \rangle
$$ 
and 
$$
\langle \phi - \phi'', v \rangle = \langle \phi, v - v_{xx} \rangle = 
\langle \phi, p \rangle.
$$
The linearized equations	(\ref{CH-lin}) and (\ref{CHm-lin-final}) are equivalent by Lemma \ref{lem-equivalency}.
\end{proof}

\begin{remark}
	The two orthogonality conditions in (\ref{orth-cond}) are 
	related to the conservation of mass (\ref{Mu}) and energy
	(\ref{Eu}) by adding a perturbation of $v$ to the smooth periodic 
	wave $\phi$ and truncating the quadratic terms in $v$. The third orthogonality condition related to the higher-order energy (\ref{Fu}) 
	is redundant due to the other two conditions:
	\begin{equation}
	\label{orth-cond-2}
	\langle \frac{3}{2} \phi^2 - \phi \phi'' - \frac{1}{2} (\phi')^2, v \rangle 
	= c \langle \phi - \phi'', v \rangle - b \langle 1, v \rangle = 0,
	\end{equation}
	where the second-order equation (\ref{CHode}) has been used.
\end{remark}

The following lemma together with Lemma \ref{lem-equivalency}
gives the sufficient condition for  
spectral stability of the periodic wave in Definition 
\ref{defstab-spectral}. 

\begin{lemma}
	\label{lem-stab-2}
	Let $\mathcal{K}|_{Y_0}$ be the restriction of $\mathcal{K}$ on $Y_0 \subset L^2_{\rm per}$. If 
\begin{equation}
\label{condition-positivity}
\mathcal{K}|_{Y_0} \geq 0 \quad \mbox{\rm and} \quad  
{\rm ker}(\mathcal{K}|_{Y_0}) = {\rm ker}(\mathcal{K}), 
\end{equation}
	then the spectrum of $(c-\phi)^{-1} \partial_x (c-\phi)^{-1}  \mathcal{K}$ in $L^2_{\rm per}$ is located on the imaginary axis. 
\end{lemma}

\begin{proof}
	Consider the spectral problem 
	\begin{equation}
	\label{stab-prob}
	J_{\phi} \mathcal{K} p = \lambda p, \quad p \in H^1_{\rm per},
	\end{equation}
	where $J_{\phi} := (c-\phi)^{-1} \partial_x (c-\phi)^{-1}$ satisfies 
	$J_{\phi}^* = - J_{\phi}$ in $L^2_{\rm per}$. The spectrum 
	of $J_{\phi} \mathcal{K}$ is purely discrete due to compact	embedding of $H^1_{\rm per}$ into $L^2_{\rm per}$.
	
	By Corollary \ref{cor-stab-1}, if $\lambda_0$ is an eigenvalue of $J_{\phi} \mathcal{K}$ in $L^2_{\rm per}$ and $\lambda_0 \neq 0$, then the corresponding eigenfunction $p_0$ satisfies 
$p_0 \in H^1_{\rm per} \cap Y_0$. A simple computation 
shows that for this $p_0 \in H^1_{\rm per} \cap Y_0$
\begin{equation*}
\lambda_0 \langle \mathcal{K} p_0, p_0 \rangle = 
\langle \mathcal{K} J_{\phi} \mathcal{K} p_0, p_0 \rangle = 
-\langle \mathcal{K} p_0, J_{\phi} \mathcal{K} p_0 \rangle 
= - \bar{\lambda}_0 \langle \mathcal{K} p_0, p_0 \rangle, 
\end{equation*}
so that 
\begin{equation*}
(\lambda_0 + \bar{\lambda}_0) \langle \mathcal{K} p_0, p_0 \rangle = 0.
\end{equation*}
Since $p_0 \in H^1_{\rm per} \cap Y_0$, then $\langle \mathcal{K} p_0, p_0 \rangle = 0$ if and only if $p_0 \in {\rm ker}(\mathcal{K})$ due to assumptions of the lemma. 
However, this is a contradiction with $\lambda_0 \neq 0$. Hence, 
$\langle \mathcal{K} p_0, p_0 \rangle > 0$, which implies that 
$\lambda_0 \in i \mathbb{R}$. This proves the assertion of the lemma.
\end{proof}

For the proof of spectral stability in Theorem \ref{theorem-stability}, 
it remains to justify the sufficient condition (\ref{condition-positivity}) 
for the operator $\mathcal{K}$. The following proposition from Theorem 4.1 in \cite{Pel-book} formulates the useful result.

\begin{proposition}\cite{Pel-book}
	\label{prop-Pel} Let $L$ be a self-adjoint operator in a Hilbert space $H$ with the inner product $\langle \cdot, \cdot \rangle$ such that $L$ has $n(L)$ negative eigenvalues (counting their multiplicities) 
	and $z(L)$ multiplicity of the zero eigenvalue bounded away from the positive spectrum of $L$. Let $\{ v_j \}_{j = 1}^N$ be a linearly independent set in $H$ and define
	$$
	H_0 := \{ f \in H : \quad \langle f, v_1 \rangle = 
	\langle f, v_2 \rangle = \dots = \langle f, v_N \rangle = 0 \}.
	$$ 
	Let $A(\lambda)$ be the matrix-valued function defined by its elements
	$$
	A_{ij}(\lambda) := \langle (L - \lambda I)^{-1} v_i, v_j \rangle, \quad 1 \leq i,j \leq N, \quad \lambda \notin \sigma(L).
	$$
	Then, 
	\begin{equation}
	\label{count-neg}
	\left\{ \begin{array}{l}
	n(L \big|_{H_0}) = n(L) - n_0 - z_0, \\
	z(L \big|_{H_0})= z(L) + z_0 - z_{\infty},
	\end{array} \right.
	\end{equation}
	where $n_0$, $z_0$, and $p_0$ are the numbers of negative, zero, and positive eigenvalues of $\lim_{\lambda \uparrow 0} A(\lambda)$ (counting their multiplicities) and $z_{\infty} = N - n_0 - z_0 - p_0$ is the number of eigenvalues of $A(\lambda)$ diverging in the limit $\lambda \uparrow 0$.
\end{proposition}

By Lemma \ref{lemma-fixed-period}, for a fixed $c > 0$ and $L > 0$, 
there exists a $C^1$ mapping $a \mapsto b = \mathcal{B}_L(a)$ and 
a $C^1$ mapping $a \mapsto \phi = \Phi_L(\cdot,a) \in H^{\infty}_{\rm per}$ 
of smooth $L$-periodic solutions along the curve $b = \mathcal{B}_L(a)$. 
Along this curve we define 
\begin{equation}
\label{mass-energy-along-the-curve}
\mathcal{M}_L(a) := M(\Phi_L(\cdot,a)) \quad \mbox{\rm and} \quad 
\mathcal{E}_L(a) := E(\Phi_L(\cdot,a)),
\end{equation}
where $M(u)$ and $E(u)$ are given by (\ref{Mu}) and (\ref{Eu}). 
In order to include the dependence on $c$, we will now write 
$\mathcal{M}_L(a,c)$ and $\mathcal{E}_L(a,c)$. 
The following lemma provides the criterion for positivity of $\mathcal{K}|_{Y_0}$ 
based on Proposition \ref{prop-Pel}. 

\begin{lemma}
	\label{lem-positivity}
	For fixed $c > 0$ and $L > 0$, the condition (\ref{condition-positivity}) 
	is satisfied if and only if 
	\begin{equation}
	\label{stability-criterion-again}
	\frac{d}{da} \frac{\mathcal{E}_L(a)}{\mathcal{M}_L(a)^2} < 0
	\end{equation}
	along the curve $b = \mathcal{B}_L(a)$.
\end{lemma}

\begin{proof}
	Since $\Phi_L(\cdot,a) \in H^{\infty}_{\rm per}$ is also $C^1$ with respect to $c$ as follows from the scaling transformation (\ref{scal-transform}), 
	we are allowed to differentiate the second-order equation 
	(\ref{second-order}) in $a$ and $c$. Writing this equation as 
	$(c-\phi)^3 \mu = a (c-\phi)$ for $\mu = \phi - \phi''$ 
	and differentiating it in $a$ and $c$, we obtain 
\begin{equation}
\label{relationphicb-again}
\mathcal{K} \partial_a \mu = c-\phi, \quad 
\mathcal{K} \partial_c \mu = -2a.
\end{equation} 
Since $a > 0$, we express
$$
\mathcal{K}^{-1} 1 = -\frac{1}{2a} \partial_c \mu, \quad 
\mathcal{K}^{-1} \phi =  -\partial_a \mu - \frac{c}{2a} \partial_c \mu,
$$
By Proposition \ref{prop-Pel}, we construct the bounded $2$-by-$2$ matrix 
$P := \lim\limits_{\lambda \uparrow 0} A(\lambda)$ in
\begin{equation}
\label{matrix-P-again}
P = \left[ \begin{matrix}\langle \mathcal{K}^{-1} 1, 1 \rangle & 
\langle \mathcal{K}^{-1} \phi, 1 \rangle \\
\langle \mathcal{K}^{-1} 1 , \phi\rangle & 
\langle \mathcal{K}^{-1} \phi, \phi \rangle \end{matrix} \right] 
= \left[ \begin{matrix} -\frac{1}{2a}\partial_c \mathcal{M}_L & 
- \partial_a \mathcal{M}_L -\frac{c}{2a} \partial_c \mathcal{M}_L \\
-\frac{1}{2a} \partial_c \mathcal{E}_L & -\partial_a \mathcal{E}_L 
-\frac{c}{2a} \partial_c \mathcal{E}_L \end{matrix} \right],
\end{equation}
where $\mathcal{M}_L$ and $\mathcal{E}_L$ in (\ref{mass-energy-along-the-curve}) are $C^1$ functions in $a$ and $c$. It follows from (\ref{matrix-P-again}) that 
\begin{eqnarray}
\label{det}
{\rm det}(P) = \frac{1}{2a} \left[ \partial_c \mathcal{M}_L \partial_a \mathcal{E}_L
- \partial_a \mathcal{M}_L \partial_c \mathcal{E}_L \right].
\end{eqnarray}
By using the scaling transformation (\ref{scal-transform}), we write
\begin{equation}
\label{scaling-2}
\mathcal{M}_L(a,c) = c \mathcal{\hat{M}}_L(\alpha), \quad \mathcal{E}_L(a,c) = c^2 \mathcal{\hat{E}}_L(\alpha), \quad a = c^3 \alpha,
\end{equation}
where $\mathcal{\hat{M}}_L$ and $\mathcal{\hat{E}}_L$ can be computed 
by formally setting $c = 1$. Substituting the transformation (\ref{scaling-2}) into (\ref{det}), we obtain
\begin{eqnarray}
\nonumber
\det(P) &=& \frac{1}{2\alpha c^4} 
\left[ \mathcal{\hat{M}}_L(\alpha) \mathcal{\hat{E}}_L'(\alpha) - 2\mathcal{\hat{E}}_L(\alpha) \mathcal{\hat{M}}_L'(\alpha) \right] \\
\label{det-P-again}
&=& 
\frac{\mathcal{\hat{M}}_L^3(\alpha)}{2\alpha c^4}  \frac{d}{d \alpha} \left( \frac{\mathcal{\hat{E}}_L(\alpha)}{\mathcal{\hat{M}}_L(\alpha)^2} \right).
\end{eqnarray}
Thus, $\det(P) < 0$ if and only if the condition (\ref{stability-criterion-again}) is satisfied for a given $c > 0$. 
Since $n(\mathcal{K}) = 1$ and $z(\mathcal{K}) = 1$ by Theorem \ref{theorem-existence} independently of $a$ and $c$, we use the 
count formulas (\ref{count-neg}) to get 
$n(\mathcal{K}|_{Y_0}) = 0$ and $z(\mathcal{K}|_{Y_0}) = 1$ 
since $n_0 = 1$, $z_0 = z_{\infty} = 0$. Hence, the conditions (\ref{condition-positivity}) are satisfied if and only if the condition (\ref{stability-criterion-again}) is satisfied.
\end{proof}

Numerical results show that the condition (\ref{stability-criterion-again}) 
is satisfied for every $c > 0$ and $L > 0$ along the curve $b = \mathcal{B}_L(a)$ 
for $a \in (0,a_L)$ in Lemma \ref{lemma-fixed-period}. 
Figure \ref{fig-dependence} shows that the mapping $a \mapsto \frac{\mathcal{E}_L(a)}{\mathcal{M}_L^2(a)}$
is monotonically decreasing for four values of $L$. 
The numerical method used to generate Figure \ref{fig-dependence} 
is described in Appendix \ref{appendix}.

\begin{figure}[htb!]
	\includegraphics[width=0.6\textwidth]{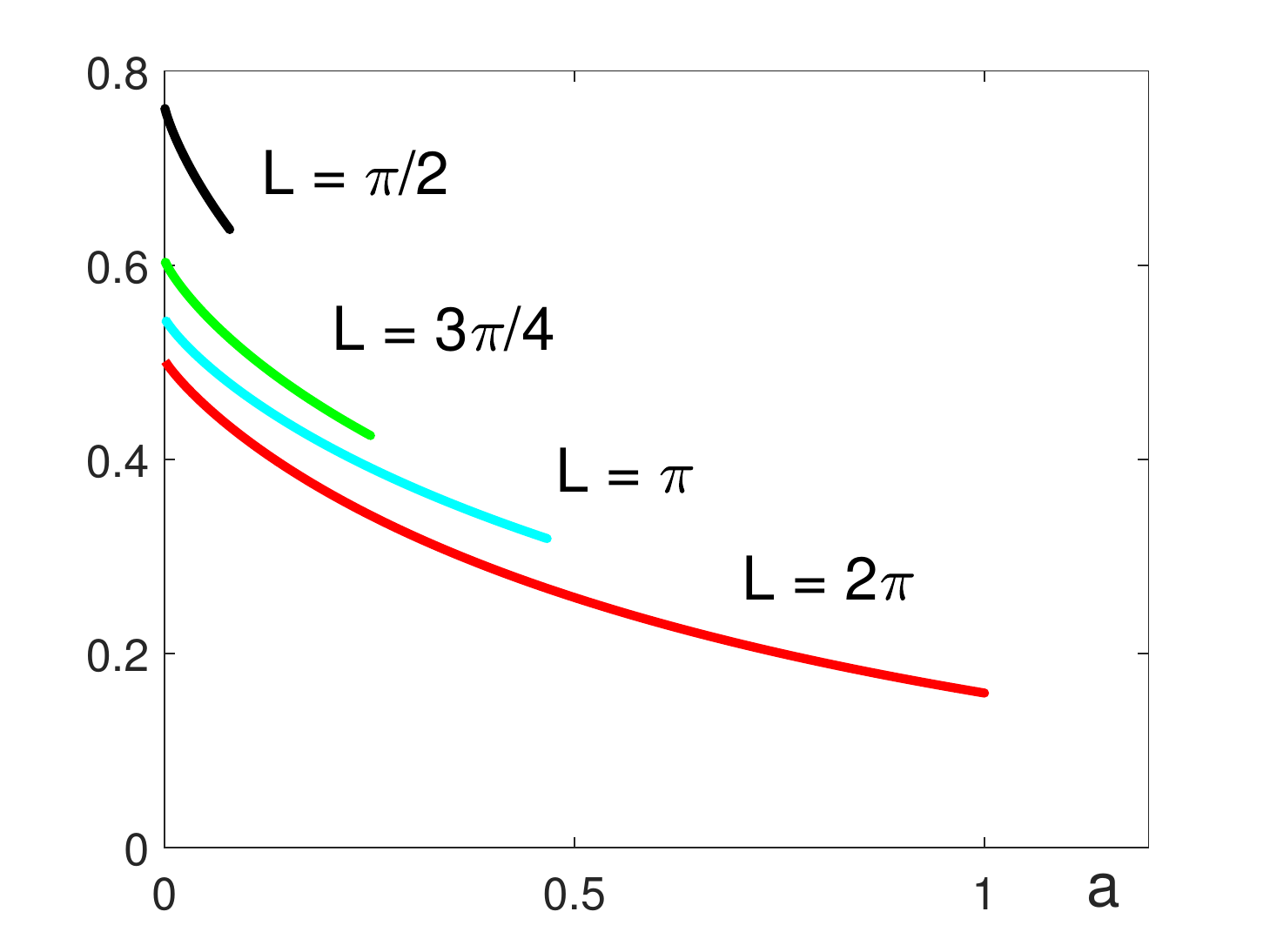}
	\caption{The dependence of $\mathcal{E}_L/\mathcal{M}_L^2$ versus $a$ along the curve $b = \mathcal{B}_L(a)$ for $c = 2$ and five values of 
		period: $L = \pi/2$ (black), $L = 3\pi/4$ (green), $L = \pi$ (cyan), and $L = 2 \pi$ (red).} 
	\label{fig-dependence}
\end{figure}

In the rest of this section, we will explain why the linearized CH equation (\ref{CH-lin}) associated with the operator $J \mathcal{L}$ is not 
convenient for the proof of spectral stability of the smooth periodic waves. 
The following lemma gives the necessary and sufficient condition 
for the $C^1$ continuation of the smooth periodic waves 
with respect to parameter $b$.

\begin{lemma}
	\label{prop-exist-surface}
	For fixed $c > 0$ and $L>0$, there exists a $C^1$ mapping 
	$b \mapsto \phi = \Psi_L(\cdot,b) \in H^{\infty}_{\rm per}$ 
	of smooth $L$-periodic solutions of Lemma \ref{period-to-energy} 
	if and only if $\partial_a \mathfrak{L} \neq 0$, where $\mathfrak{L}(a,b,c)$ is the period function. 
\end{lemma}

\begin{proof}
	If $\partial_a \mathfrak{L} \neq 0$, then arguments of the proof 
	of Lemma \ref{lemma-fixed-period} based on the implicit function 
	theorem and smoothness of smooth periodic solutions of Lemma 
	\ref{period-to-energy} with respect to parameters 
	gives existence of the $C^1$ mapping $b \mapsto \phi = \Psi_L(\cdot,b) \in H^{\infty}_{\rm per}$.
	
	In the converse direction, we assume existence of the $C^1$ mapping $b \mapsto \phi = \Psi_L(\cdot,b) \in H^{\infty}_{\rm per}$ and prove that 
	$\partial_a \mathfrak{L}\neq 0$. Due to the $C^1$ smoothness, it follows by differentiating the second-order equation (\ref{CHode}) in $c$ and $b$ that 
	\begin{equation}
	\label{relationphicb}
	\mathcal{L}\partial_c\phi=\phi''-
	\phi\ \ \ \mbox{and}\ \ \ \ \mathcal{L}\partial_b\phi=1,
	\end{equation} 
	Let $\{y_1,y_2\}$ be the fundamental set of solutions associated to the equation $\mathcal{L}v=0$ in $H^2(0,L)$ as in (\ref{fund-set}) and (\ref{y1-y2}). By Liouville's theorem, the associated Wronskian is given by 
	\begin{equation}\label{wronskian}
	\mathcal{W}(y_1,y_2)(x)=\displaystyle e^{\int_0^x\frac{\phi'(s)}{c-\phi(s)}ds}=\frac{c-\phi(0)}{c-\phi(x)}>0.
	\end{equation}
	Now, since	$\mathcal{W}(y_1,y_2)(x)=y_1(x)y_2'(x)-y_1'(x)y_2(x)$ for all $x\in \mathbb{R}$ and $y_2(x) = \phi'(x)/\phi''(0)$, we obtain by (\ref{wronskian}) that
	\begin{equation}
	\label{eq-wronsk}
	\phi''(0) \int_0^{L}\frac{c-\phi(0)}{c-\phi(x)}dx=\int_{0}^{L}
	\left[ y_1(x)\phi''(x)-y_1'(x)\phi'(x) \right] dx
	\end{equation}
	By contradiction, assume that $\partial_a \mathfrak{L} = 0$, 
	then $y_1$ is $L$-periodic similar to $y_2$. Integration by parts in (\ref{eq-wronsk}) yields 
	\begin{equation}
	\label{eq-wronsk1}
	\phi''(0) \int_0^{L}\frac{c-\phi(0)}{c-\phi(x)}dx = 2 \langle \phi'', y_1 \rangle.
	\end{equation}
	It follows from (\ref{relationphicb}) that
	\begin{equation}
	\label{eq-wronsk2}
	\langle \phi'', y_1 \rangle = \langle \phi, y_1 \rangle, \qquad 
	\langle 1, y_1 \rangle = 0.	
	\end{equation}
	On the other hand, we also have $\mathcal{L}1=c-3\phi+\phi''$, 
	hence 
	\begin{equation}
	\label{eq-wronsk3}
	\langle \phi'', y_1 \rangle - 3 \langle \phi, y_1 \rangle + c \langle 1, y_1 \rangle = 0.
	\end{equation}
	Substituting (\ref{eq-wronsk2}) into (\ref{eq-wronsk3}) yields $\langle \phi'', y_1 \rangle = 0$, which is a contradiction with the non-zero left-hand side in (\ref{eq-wronsk1}). Hence, $\partial_a \mathfrak{L} = 0$ leads to the contradiction with the $C^1$ smoothness of the mapping $b \mapsto \phi = \Psi_L(\cdot,b) \in H^{\infty}_{\rm per}$. 
\end{proof}

\begin{remark}
	\label{remark-limitations}
By Lemma \ref{prop-exist-surface}, the relation $\mathcal{L}\partial_b \phi = 1$ in (\ref{relationphicb}) cannot be used at the points where $\partial_a \mathfrak{L} = 0$.
Away from these points, the $2$-by-$2$ matrix of projections in 
Proposition \ref{prop-Pel} can be constructed and evaluated 
for the operator $\mathcal{L}$ under the two orthogonality conditions 
in $X_0$ given by (\ref{orth-cond}) as follows:
\begin{eqnarray}
\label{matrix-P}
S := \left[ \begin{matrix}\langle \mathcal{L}^{-1} 1, 1 \rangle & 
\langle \mathcal{L}^{-1} (\phi - \phi''), 1 \rangle \\
\langle \mathcal{L}^{-1} 1, (\phi - \phi'') \rangle & 
\langle \mathcal{L}^{-1} (\phi - \phi''), (\phi - \phi'') \rangle \end{matrix} \right] = \left[ \begin{matrix} \partial_b \mathcal{M}_L & -\partial_b \mathcal{E}_L \\
\partial_c \mathcal{M}_L & -\partial_c \mathcal{E}_L  \end{matrix} \right],
\end{eqnarray}
where $\mathcal{M}_L$ and $\mathcal{E}_L$ are computed at $\phi = \Psi_L(\cdot,b) \in H^{\infty}_{\rm per}$ and extended in both $b$ and $c$. By using the scaling transformation (\ref{scal-transform}), we write
\begin{equation}
\label{scaling-1}
\mathcal{M}_L(c,b) = c \hat{\mathcal{M}}_L(\beta), \quad \mathcal{E}_L(c,b) = c^2 \hat{\mathcal{E}}_L(\beta), \quad b = c^2 \beta.
\end{equation}
Substituting the transformation (\ref{scaling-1}) into (\ref{matrix-P}) yields
\begin{equation}
\label{det-P}
\det(S) = \hat{\mathcal{M}}_L(\beta) \hat{\mathcal{E}}_L'(\beta) - 2 \hat{\mathcal{E}}_L(\beta) \hat{\mathcal{M}}_L'(\beta) = 
\hat{\mathcal{M}}_L^3(\beta) \frac{d}{d \beta} \left( \frac{\hat{\mathcal{E}}_L(\beta)}{\hat{\mathcal{M}}_L(\beta)^2} \right).
\end{equation}
Here the derivative is computed along the curve $b = \mathcal{B}_L(a)$, 
where $\mathcal{B}_L'(a) > 0$ if $n(\mathcal{L}) = 1$ 
and $\mathcal{B}_L'(a) < 0$ if $n(\mathcal{L}) = 2$, 
see Figure \ref{fig-domain}.
In the former case, our numerical results show that 
$\det(S) < 0$ so that $n_0 = 1$, $z_0 = z_{\infty} = 0$ 
and by Proposition \ref{prop-Pel}, we have $n(\mathcal{L}|_{X_0}) = 0$ 
and $z(\mathcal{L}|_{X_0}) = 1$. In the latter case, our numerical 
results give $\det(S) > 0$ and $\hat{\mathcal{M}}_L'(\beta) < 0$  
	so that the $2$-by-$2$ matrix $S$ is negative with $n_0 = 2$, $z_0 = z_{\infty} = 0$ 
	and by Proposition \ref{prop-Pel}, we still have $n(\mathcal{L}|_{X_0}) = 0$ 
	and $z(\mathcal{L}|_{X_0}) = 1$. 
At the points where $\partial_a \mathfrak{L}= 0$, $S$ is unbounded with $n_0 = 1$, $z_0 = 1$ and $z_{\infty} = 1$ which still gives $n(\mathcal{L}|_{X_0}) = 0$ 
and $z(\mathcal{L}|_{X_0}) = 1$ since $n(\mathcal{L}) = 1$ and 
$z(\mathcal{L}) = 2$. Thus, we obtain conditions 
\begin{equation}
\label{condition-positivity-L}
\mathcal{L}|_{X_0} \geq 0 \quad \mbox{\rm and} \quad  
{\rm ker}(\mathcal{L}|_{X_0}) = {\rm ker}(\mathcal{L})
\end{equation}
for the linearized operator $\mathcal{L}$ but with three different 
computations depending on whether $\partial_a \mathfrak{L} < 0$, $\partial_a \mathfrak{L} = 0$, and $\partial_a \mathfrak{L} > 0$.
\end{remark}

\section{Orbital stability of periodic waves}
\label{sec-5}

Here we prove the orbital stability of periodic waves stated in Theorem \ref{theorem-stability}. We follow the approach in \cite{ANP}, where the following useful result was proven in Proposition 3.8 and Theorem 4.2.

\begin{proposition}\cite{ANP}
	\label{prop-Natali}
	Let $V(u)$ be a conserved quantity in the time evolution 
	of the Hamiltonian system (\ref{sympl-1}). Assume that the linearized operator $\mathcal{L}$ at the periodic travelling wave with profile $\phi$ 
	admits a simple negative and a simple zero eigenvalue 
	with ${\rm Ker}(\mathcal{L}) = {\rm span}(\phi')$ satisfying $\langle V'(\phi), \phi' \rangle = 0$. 
	Assume that there exists $\Upsilon \in H^2_{\rm per}$ such that 
	$\langle \mathcal{L}\Upsilon, v \rangle = 0$ for every $v \in L^2_{\rm per}$ such that $\langle V'(\phi), v \rangle = 0$. If $\langle \mathcal{L}\Upsilon, \Upsilon \rangle < 0$, then the periodic travelling wave is orbitally stable in the time evolution of  (\ref{sympl-1}) in $H^1_{\rm per}$.
\end{proposition}

\begin{remark}
	\label{remark-orbital}
	The notion of orbital stability in Definition \ref{defstab} prescribes the existence of global solutions $u \in C(\mathbb{R},H^s_{\rm per})$ for $s > \frac{3}{2}$. The local solutions $u \in C((-t_0,t_0),H^s_{\rm per})$ for some $t_0 > 0$ exist due to the local well-posedness 
	theory in \cite{CE-1998,hakka,HM2002}. Since $M$, $E$, and $F$ are conserved quantities, one can combine the local solution  with the standard a priori estimates $M(u(t)) = M(u_0)$, $E(u(t))=E(u_0)$, and $F(u(t)) = F(u_0)$ for all $t\geq0$ in order to extend the local to global solutions near the smooth periodic waves in the case when they are stable by Proposition \ref{prop-Natali}.
\end{remark}

\begin{remark}
	\label{rem-Natali}
	The result of Proposition \ref{prop-Natali} can be equivalently written for the Hamiltonian system (\ref{sympl-2}) with the conserved quantity $V(m)$ written in variable $m = u - u_{xx}$ and with the linearlized operator $\mathcal{K}$ at the periodic travelling wave $\mu = \phi - \phi''$.
\end{remark}

The following lemma transfers the spectral stability criterion in Lemma \ref{lem-positivity} to the orbital stability criterion.

\begin{lemma} 
	\label{lemma-orbital}
	For fixed $c > 0$ and $L > 0$, the smooth $L$-periodic wave with profile
	  $\phi = \Phi_L(\cdot,a) \in H^{\infty}_{\rm per}$ is orbitally stable in $H^1_{\rm per}$ 
	  if the mapping
	\begin{equation}
		\label{stability-criterion1}
		a \mapsto \frac{\mathcal{E}_L(a)}{\mathcal{M}_L(a)^2}
	\end{equation}
	is strictly decreasing along the curve $b = \mathcal{B}_L(a)$.
	\end{lemma}

\begin{proof}
 For $\mu=\phi-\phi''$, let us rewrite (\ref{relationphicb-again}) in the form:
	\begin{equation}\label{relation-K-manifold}
		\mathcal{K}\left(\frac{1}{2a} \partial_c\mu\right)=-1,\ \ \ \mathcal{K}\left(\partial_a\mu+\frac{c}{2a}\partial_c\mu\right)=-\phi.
		\end{equation}
For $m = u - u_{xx}$, we define a linear superposition 
of the two conserved quantities (\ref{Mu}) and (\ref{Eu}):
\begin{equation}
\label{lin-sup-cons}
 V(m) := r M(u)+sE(u), \quad u = (1-\partial_x^2)^{-1} m,
\end{equation}
where $r$ and $s$ are real coefficients, $M(u)$ is given by 
(\ref{Mu}), and $E(u)$ is given by (\ref{Em}).
Since $V'(\mu)=r+s\phi$ and 
${\rm Ker}(\mathcal{K}) = {\rm span}(\mu')$, we check that 
$\langle V'(\mu),\mu'\rangle=0$ since $\phi \in H^{\infty}_{\rm per}$. 
By Theorem \ref{theorem-existence}, the linearized operator 
$\mathcal{K}$ satisfies the assumption of Proposition \ref{prop-Natali}. 
We then proceed by constructing $\Upsilon$. Letting 
$$
\Upsilon := \frac{r}{2a} \partial_c\mu + s \left(\partial_a\mu+\frac{c}{2a}\partial_c\mu\right),
$$
it follows from (\ref{relation-K-manifold}) that 
$\mathcal{K} \Upsilon = -r-s\phi$ and 
$\langle\mathcal{K} \Upsilon, p \rangle=0$, for all $p \in Y_0$ 
defined in (\ref{orth-cond-again}).
A straightforward calculation gives us that	
\begin{equation}\label{quadraticK}
	\langle\mathcal{K} \Upsilon,\Upsilon\rangle = r^2\langle \mathcal{K}^{-1}1,1\rangle+2rs\langle \mathcal{K}^{-1}\phi,1\rangle+s^2\langle \mathcal{K}^{-1}\phi,\phi\rangle.
	\end{equation}
The quadratic form (\ref{quadraticK}) in $r$ and $s$ is defined by 
the same $2$-by-$2$ symmetric matrix $P$ as in (\ref{matrix-P-again}).  If condition (\ref{stability-criterion1}) is satisfied, we have that $\det(P) < 0$ and there exists a choice of real coefficients $r$ and $s$ such that $	\langle\mathcal{K} \Upsilon,\Upsilon\rangle < 0$. Hence, the orbital stability 
of the periodic waves in the time evolution of the Hamiltonian system 
(\ref{sympl-2}) follows from Proposition \ref{prop-Natali} and Remark \ref{rem-Natali}.
\end{proof}

In what follows, we show that the orbital stability condition 
is satisfied for every $b \leq 0$. The main advantage of this result is that we do not need to verify the criterion 
(\ref{stability-criterion1}) by using numerical computations. The following lemma 
reports the relevant result.

\begin{lemma}
	\label{lemma-b-neg} 
	For fixed $c > 0$ and $b \leq 0$, the smooth periodic wave 
with profile $\phi \in H^{\infty}_{\rm per}$ is orbitally stable in $H^1_{\rm per}$.
\end{lemma}

\begin{proof}
	It follows from (\ref{second-order}) and (\ref{hill-m}) that
	\begin{equation}
	\label{Lmu1}
	\mathcal{K} \mu = (c-\phi)^3 \mu - 2a (1-\partial_x^2)^{-1} \mu = a(c-3\phi).
	\end{equation}
We can define $V(m) := c M(u) - 3 E(u)$ so that if 
$\langle V'(\mu), p \rangle = 0$, then $\langle\mathcal{K}\mu, p\rangle=0$. Thus, we can take $\Upsilon := \mu$ and compute
	\begin{equation}\label{Lmu2}
	\langle\mathcal{K}\mu,\mu\rangle = a \left[ c M(\phi) - 6 E(\phi) \right].
	\end{equation}
Since $a > 0$, we check the sign of $c M(\phi) - 6 E(\phi)$:
\begin{eqnarray*}
c M(\phi) - 6 E(\phi) &=& \int_{0}^L \left[ c \phi - 3 \phi^2 - 3 (\phi')^2 \right] dx
\\
&=&  \int_{0}^L \left[ 2b + c \phi'' - c \phi -2 (\phi')^2 \right] dx	\\
\\
&=& 2b L - c M(\phi) - 2 \int_{0}^L  (\phi')^2 dx,
\end{eqnarray*}
where we have used $\mathcal{L}1 = c - 3 \phi + \phi''$ and $\mathcal{L}\phi = 2b + c(\phi'' - \phi)$ and $\langle \mathcal{L}1, \phi \rangle = \langle 1, \mathcal{L} \phi \rangle$. If $b \leq 0$ and $c > 0$, then $\langle\mathcal{K}\mu,\mu\rangle  < 0$ since $M(\phi) > 0$, 
so that the periodic waves are orbitally in the time evolution of the Hamiltonian system (\ref{sympl-2}) by Proposition 
\ref{prop-Natali} and Remark \ref{rem-Natali}.
\end{proof}

\begin{remark}
	\label{remark-boundaries}
	The criterion $\langle \mathcal{K}\mu, \mu \rangle < 0$ of the orbital stability is satisfied near the boundary $a = a_-(b)$ in Lemma \ref{remark-right} both for $b \leq 0$ and $b > 0$. Indeed, we can write 
	$$
	\langle \mathcal{K}\mu, \mu \rangle = a^2 \int_0^L \frac{c-3\phi}{(c-\phi)^2} dx,
	$$
	where $c - 3 \phi < 0$ because the constant solution $\phi= \phi_2$ satisfies the ordering (\ref{ordering}). 
\end{remark}

\begin{remark}
The criterion $\langle \mathcal{K}\mu, \mu \rangle < 0$ is not satisfied near the boundary $a = a_+(b)$ in Lemma \ref{remark-upper} for $b > 0$. Indeed, since $\phi = \phi_1 + \hat{\phi}$, where $\hat{\phi}(x) \to 0$ as $|x| \to \infty$, we derive 
	$$
	c M(\phi) - 6 E(\phi) = L \phi_1 (c - 3 \phi_1) + \mathcal{O}(1),
	$$
	where $\mathcal{O}(1)$ denotes bounded terms in the limit $L \to \infty$. 
	Since $c - 3 \phi_1 > 0$ by the ordering (\ref{ordering}) and $a > 0$, 
	we have $\langle \mathcal{K}\mu, \mu \rangle > 0$ near the boundary $a = a_+(b)$. Thus, the criterion for orbital stability in Lemma \ref{lemma-b-neg} is not as sharp as the criterion in Lemma \ref{lemma-orbital}. 
\end{remark}

\begin{remark}
	\label{teoest} 
The result of Lemma \ref{lemma-b-neg} can be established directly for the linearized operator $\mathcal{L}$. 	Since $\mathcal{L}\phi = 2b + c (\phi'' - \phi)$, we can define $V(u) := 2b M(u) - c E(u)$ so that if 
$\langle V'(u), v \rangle = 0$, then $\langle\mathcal{L}\phi,v\rangle=0$. Thus, we can take $\Upsilon := \phi$ and compute 
	\begin{equation}\label{Lphi3}
	\langle\mathcal{L} \phi,\phi \rangle = 2b M(\phi) - 2 c E(\phi).
	\end{equation}
If $b \leq 0$ and $c > 0$, then $\langle \mathcal{L}\phi, \phi \rangle < 0$ 
since $M(\phi) > 0$. However, Proposition \ref{prop-Natali} can only be used if $\mathcal{L}$ has a simple negative eigenvalue, which is only true in a subset 
of $b \leq 0$, where $\partial_a \mathfrak{L} < 0$. Similar to Remark \ref{remark-limitations}, we can see that the linearized operator $\mathcal{K}$ provides wider region for orbital stability compared to the linearized operator $\mathcal{L}$.
\end{remark}

\section{Conclusion}
\label{sec-6}

We have studied spectral and orbital stability of smooth periodic travelling waves in the Camassa--Holm (CH) equation by using functional-analytic tools. We showed that the standard Hamiltonian formulation of the CH equation has several shortcomings, e.g. the number of negative eigenvalues in the linearized operator is either one or two depending on the parameters of the periodic travelling wave. On the other hand, the nonstandard Hamiltonian formulation based on the momentum quantity $m := u - u_{xx}$ provides a better framework for analysis with only one simple negative eigenvalue of the associated linearized operator. 

The criterion for spectral and orbital stability has been derived by using the nonstandard Hamiltonian formulation. The stability criterion has been checked numerically and it is an open problem to prove analytically that this criterion is satisfied in the entire existence region for the smooth periodic travelling waves. We proved analytically that the stability criterion is satisfied in a subset of the existence region. 

Since the CH equation is a prototypical example of a more general 
class of nonlinear evolution equations, it is expected that our methods 
will be useful for the analysis of spectral and orbital stability in the systems 
where other methods based on the inverse scattering transform are not applicable, e.g., for the $b$-family of the CH equations. 

\appendix
\section{Smooth periodic waves in the explicit form}
\label{appendix}

Here we derive the explicit expressions for the smooth periodic wave with the profile $\phi$ satisfying the system (\ref{CHode}), (\ref{second-order}), and (\ref{quadra}). Since $\phi < c$, we can transform the variables 
\begin{equation}
\label{CH-KDV}
\phi(x)=\psi(z(x)), \quad
z(x)=\int_0^x\frac{ds}{\sqrt{c-\phi(s)}}
\end{equation}
and rewrite the second-order equation (\ref{CHode}) with the chain rule to the form
\begin{equation}
\label{KDVode}
-\psi'' + c \psi -\frac{3}{2} \psi^2 = b,
\end{equation}
which is the stationary KdV equation. Note that this reduction of the 
travelling periodic waves of the CH equation (\ref{CH}) to the travelling periodic waves of the KdV equation is different from the previously explored connection 
between the CH and KdV equations in \cite{LenCHkdv}. A similar transformation 
was used in \cite{CS2} in the context of the solitary waves at the zero background. 

The second-order equation (\ref{KDVode}) has the following explicit periodic solution (see, e.g., \cite{NLP}):
\begin{equation}
\label{exact-KdV}
\psi(z) = \frac{1}{3} c + \frac{4}{3} \gamma^2 \left[ 1 - 2 k^2 + 3 k^2 {\rm cn}^2(\gamma z;k) \right],
\end{equation}
where $\gamma > 0$ and $k \in (0,1)$ are arbitrary parameters, 
and ${\rm cn}$ is the Jacobian elliptic function. 
The period of the periodic solution $\psi$ in $z$ is 
$P = 2 \gamma^{-1} K(k)$. The free parameters $\gamma$ and $k$ 
parametrize the turning points $\phi_-$ and $\phi_+$ in (\ref{ordering-turning-points}) and (\ref{roots-turning-points}):
\begin{equation}
\label{turning-points-KdV}
\left\{ \begin{array}{l} \phi_+ = \frac{1}{3} c + \frac{4}{3} \gamma^2 (1 + k^2), \\ \phi_- = \frac{1}{3} c + \frac{4}{3} \gamma^2 (1 - 2 k^2), \end{array} \right.
\end{equation}
which can be inverted as follows 
\begin{equation}
\label{turning-points-KdV-inverse}
\gamma^2 = \frac{1}{4} (2 \phi_+ + \phi_- - c), \quad
k^2 = \frac{\phi_+ - \phi_-}{2 \phi_+ + \phi_- - c}.
\end{equation}
Parameters $a$ and $b$ are related to parameters $\gamma$ and $k$ by 
\begin{align}
\nonumber
a &= \frac{1}{2} (\phi_+ + \phi_-) (c - \phi_+) (c - \phi_-) \\
&= \frac{4}{27} (c + 2 \gamma^2(2-k^2)) (c - 2 \gamma^2(1+k^2)) (c - 2 \gamma^2(1-2 k^2)) 
\label{appendix-a}
\end{align}
and
\begin{align}
\nonumber
b &= \frac{1}{2} c (\phi_+ + \phi_-) - \frac{1}{2} 
(\phi_+^2 + \phi_+\phi_- + \phi_-^2) \\
&= \frac{1}{6} c^2 - \frac{8}{3} \gamma^4 (1 - k^2 + k^4),
\label{appendix-b}
\end{align}
where equations (\ref{roots-turning-points}) have been used. When $a$ or $b$ are fixed, e.g., for numerical results obtained 
on Figs. \ref{fig-graphs-b} and \ref{fig-graphs}, we express $\gamma$ 
from the roots of either (\ref{appendix-a}) or (\ref{appendix-b}) 
and parameterize the family by the only parameter $k$ in a subset of $(0,1)$.

The period function $L = \mathfrak{L}(a,b,c)$ can be expressed by (\ref{CH-KDV}) and (\ref{exact-KdV}) in the form:
\begin{equation}
\label{period-KdV}
L = \frac{\sqrt{2}}{\gamma \sqrt{3}} \int_0^{2K(k)} \sqrt{c - 2 \gamma^2 (1 - 2 k^2 + 3k^2 {\rm cn}^2(z;k))} dz.
\end{equation}
If $L$ is fixed, e.g., for numerical results obtained on Figs. \ref{fig-domain} and \ref{fig-dependence}, 
then $\gamma$ can be found from a root finding algorithm for equation (\ref{period-KdV}), after which the periodic solutions are parameterized 
by the only parameter $k$ in a subset of $(0,1)$. The mass and energy integrals 
in (\ref{Mu}) and (\ref{Eu}) are evaluated at the periodic wave (\ref{CH-KDV}) with the chain rule:
\begin{equation}
M(\phi) = \int_0^{2 \gamma^{-1} K(k)} \psi(z) \sqrt{c - \psi(z)} dz
\end{equation}
and 
\begin{equation}
E(\phi) = \int_0^{2 \gamma^{-1} K(k)} \left[ b + \psi(z)^2 - \frac{a}{c - \psi(z)} \right] \sqrt{c - \psi(z)} dz,
\end{equation}
where we have used $(\phi')^2 = \phi^2 + 2 b - 2a/(c-\phi)$.

The limit $k \to 0$ corresponds to the constant solution 
\begin{equation}
\psi = \frac{1}{3} c + \frac{4}{3} \gamma^2
\end{equation}
in Lemma \ref{remark-right}. The limit $k \to 1$ corresponds to the solitary wave solution 
\begin{equation}
\label{soliton-limit}
\psi(z) = \frac{1}{3} c + \frac{4}{3} \gamma^2 \left[ -1 + 3 {\rm sech}^2(\gamma z) \right]
\end{equation}
in Lemma \ref{remark-upper}. The curve $c = 2 \gamma^2 (1+k^2)$ corresponds to the peaked periodic wave
\begin{equation}
\psi(z) = 2 \gamma^2 \left[ 1 - k^2 + 2 k^2 {\rm cn}^2(\gamma z;k) \right],
\end{equation}
in Lemma \ref{remark-left}.

\end{document}